\documentclass[10pt, reqno]{amsart}
\usepackage{amssymb}
\usepackage{hyperref}
\allowdisplaybreaks

\numberwithin{equation}{section}

\newtheorem{theorem}{Theorem}[section]

\newtheorem{lemma}[theorem]{Lemma}

\newtheorem{corollary}[theorem]{Corollary}

\newtheorem{proposition}[theorem]{Proposition}

\theoremstyle{definition}
\newtheorem{definition}[theorem]{Definition}
\newtheorem{remark}[theorem]{Remark}

\theoremstyle{remark}

\newcommand{\C}{\mathbb{C}}
\newcommand{\E}{\mathcal{E}}

\newcommand{\K}{\mathcal{K}}

\newcommand{\R}{\mathbb{R}}
\let\Re=\undefined\DeclareMathOperator*{\Re}{Re}
\let\Im=\undefined\DeclareMathOperator*{\Im}{Im}

\renewcommand{\L}{\mathcal{L}_a}

\newcommand{\eps}{\varepsilon}

\newcommand{\qtq}[1]{\quad\text{#1}\quad}

\newcounter{smalllist}

\begin{document}

\title[NLS with inverse-square potential]{The focusing cubic NLS with inverse-square potential in three space dimensions}

\author[R. Killip]{Rowan Killip}\address{Department of Mathematics, UCLA, Los Angeles, USA}\email{killip@math.ucla.edu}
\author[J. Murphy]{Jason Murphy}\address{Department of Mathematics, University of California, Berkeley, USA}\email{murphy@math.berkeley.edu}
\author[M. Visan]{Monica Visan}\address{Department of Mathematics, UCLA, Los Angeles, USA}\email{visan@math.ucla.edu}
\author[J. Zheng]{Jiqiang Zheng}\address{Universit\'e Nice Sophia-Antipolis, 06108 Nice Cedex 02, France}\email{zhengjiqiang@gmail.com, zheng@unice.fr}

\begin{abstract} We consider the focusing cubic nonlinear Schr\"odinger equation with inverse-square potential in three space dimensions.  We identify a sharp threshold between scattering and blowup, establishing a result analogous to that of Duyckaerts, Holmer, and Roudenko for the standard focusing cubic NLS \cite{DHR, HR}.  We also prove failure of uniform space-time bounds at the threshold.
\end{abstract}

\maketitle

\section{Introduction}
We consider the initial-value problem for the focusing cubic nonlinear Schr\"odinger equation (NLS) with inverse-square potential in three space dimensions.  The associated linear problem is given in terms of the operator
\[
\L := -\Delta + \tfrac{a}{|x|^2}.
\]
Restricting to values $a>-\frac14$, we consider $\L$ as the Friedrichs extension of the quadratic form $Q$, defined on $C_c^\infty(\R^3\backslash\{0\})$ via
\[
Q(f) := \int_{\R^3} |\nabla f(x)|^2 + \tfrac{a}{|x|^2}|f(x)|^2\,dx.
\]
The restriction on $a$ guarantees positivity; indeed, one has
\[
Q(f) = \int_{\R^3} |\nabla f + \tfrac{\sigma x}{|x|^2}f|^2 \,dx, \qtq{where} \sigma := \tfrac12 - \sqrt{\tfrac14 + a}.
\]

For the general theory of such extensions, we refer the reader to \cite[Section X.3]{ReedSimon}; for more on the specific operator $\L$, see \cite{KSWW, KMVZZ1}.  We choose the Friedrichs extension for the following physically-motivated reasons: (i) when $a=0$, $\L$ becomes the usual Laplacian $-\Delta$, and (ii) the Friedrichs extension appears when one takes a scaling limit of more regular potentials; for example,
\[
L_n = -\Delta + \tfrac{an^2}{1+n^2 |x|^2}\to \L \qtq{as} n\to\infty
\]
in the strong resolvent sense, where we understand $L_n$ as having domain $H_x^2(\R^3)$.

By the sharp Hardy inequality, one has
\begin{equation}\label{iso}
Q(f) = \|\sqrt{\L}\, f\|_{L_x^2}^2 \sim \|\nabla f\|_{L_x^2}^2\qtq{for} a>-\tfrac14.
\end{equation}
Thus, the Sobolev space $\dot H_x^1$ defined in terms of the gradient is isomorphic to the space $\dot H_a^1$ defined in terms of $\L$.  The paper \cite{KMVZZ1} determines the sharp range of parameters $s$ and $p$ for which 
$$
\| (\L)^{\frac{s}{2}} f\|_{L_x^p} \sim \| (-\Delta)^{\frac{s}{2}} f\|_{L_x^p} ;
$$
see Lemma~\ref{pro:equivsobolev} below.

We study the following equation:
\begin{equation}\label{nls}\tag{$\text{NLS}_a$}
(i\partial_t - \L)u = -|u|^2 u, \quad (t,x)\in\R\times\R^3
\end{equation}
for initial data in $H_x^1 (\R^3)$.  When $a=0$, \eqref{nls} reduces to the standard focusing cubic NLS:
\begin{equation}\label{nls0}\tag{$\text{NLS}_0$}
(i\partial_t+\Delta) u = -|u|^2 u.
\end{equation}

Like \eqref{nls0}, the equation \eqref{nls} enjoys several symmetries and conservation laws.  Firstly, the class of solutions is invariant under the  rescaling
\begin{equation}\label{scaling}
u(t,x) \mapsto u^\lambda(t,x) : = \lambda u(\lambda^2t, \lambda x),
\end{equation}
which identifies $\dot H_x^{\frac12}(\R^3)$ as the scaling-critical space of initial data.  Secondly, solutions to \eqref{nls} conserve their \emph{mass} and \emph{energy}, defined respectively by
\begin{align*}
& M(u(t)) := \int_{\R^3} |u(t,x)|^2 \,dx, \\
& E_a(u(t)) := \int_{\R^3} \tfrac12|\nabla u(t,x)|^2 + \tfrac{a}{2|x|^2} |u(t,x)|^2 - \tfrac14 |u(t,x)|^4 \,dx.
\end{align*}

Initial data belonging to $H_x^1(\R^3)$ have finite mass and energy, as is evident from \eqref{iso} and the following variant of the Gagliardo--Nirenberg inequality:
\begin{equation}\label{E:GN}
\|f\|_{L_x^4(\R^3)}^4 \leq C_a \|f\|_{L_x^2(\R^3)} \|f\|_{\dot H_a^1(\R^3)}^3,
\end{equation}
where $C_a$ denotes the sharp constant in the inequality above.  That $0<C_a<\infty$ follows from the standard Gagliardo--Nirenberg inequality and \eqref{iso}.  For more on \eqref{E:GN}, see Theorem~\ref{T:GN}.

In contrast to \eqref{nls0}, the equation \eqref{nls} with $a\neq 0$ is not space-translation invariant.  This introduces some of the key challenges in the development of concentration compactness tools and the induction on energy argument, as we discuss below.

We consider the problem of global existence and scattering for \eqref{nls}.  We begin with the following definitions.

\begin{definition}[Solution]\label{D:solution} Let $t_0\in\R$ and $u_0\in H_x^1(\R^3)$. Let $I$ be an interval containing $t_0$. A function $u:I\times\R^3\to\C$ is a \emph{solution} to
\[
(i\partial_t-\L)u = -|u|^2 u, \quad u(t_0)=u_0
\]
if it belongs to $C_t H_a^1\cap L_t^5 H_a^{1,\frac{30}{11}}(K\times\R^3)$ for any compact $K\subset I$ and obeys the Duhamel formula
\[
u(t) = e^{-i(t-t_0)\L}u_0 + i\int_{t_0}^t e^{-i(t-s)\L}\bigl(|u(s)|^2 u(s)\bigr)\,ds\qtq{for all}t\in I,
\]
where we rely on the self-adjointness of $\L$ to make sense of $e^{-it\L}$ via the Hilbert space functional calculus.  We call $I$ the \emph{lifespan} of $u$. We call $u$ a \emph{maximal-lifespan solution} if it cannot be extended to any strictly larger interval. If $I=\R$, we call $u$ \emph{global}.
\end{definition}

Our notion of solution relies on Sobolev spaces adapted to the linear operator $\L$, rather than traditional Fourier expansion; see the beginning of Section~\ref{S:prelim} for the definitions.  This is natural from the point of view of the underlying linear equation; in particular, powers of $\L$ commute with the propagator $e^{-it\L}$.  On the other hand, analysis of the nonlinear problem then requires a version of the Leibniz/product rule adapted to such spaces.  In Theorem~\ref{T:LWP} we show the existence of local solutions to \eqref{nls}, in the sense just described, by exploiting results from \cite{BPSTZ} and \cite{KMVZZ1}.

\begin{definition}[Scattering]  A global solution $u$ to \eqref{nls}  \emph{scatters} if there exist $u_\pm\in H_x^1(\R^3)$ such that
\[
\lim_{t\to\pm\infty} \| u(t) - e^{-it\L}u_{\pm} \|_{H_x^1(\R^3)} = 0.
\]
\end{definition}

Global existence, scattering, and blowup for \eqref{nls0} were studied in \cite{DHR, HR}.  The authors identified a sharp threshold between scattering and blowup, described in terms of the ground state $Q_0$, which is the unique, positive, radial, decaying solution to the elliptic problem
\[
\Delta Q_0 - Q_0 + Q_0^3 = 0.
\]
We state the results of \cite{DHR, HR} as Theorem~\ref{T:DHR} below. Our main result, Theorem~\ref{T:main}, is an analogous threshold result for \eqref{nls}.  To state these results, we make use of the sharp constant in the Gagliardo--Nirenberg inequality \eqref{E:GN}:
\begin{equation}
C_a := \sup\bigl\{ \ \|f\|_{L_x^4}^4 \div \bigl[\|f\|_{L_x^2} \|f\|_{\dot H_a^1}^3\bigr]:f\in H_a^1\backslash\{0\}\ \bigl\}.
\label{E:Ca}
\end{equation}

In Theorem~\ref{T:GN}, we will show (i) for $a\in (-\frac14,0]$, equality is attained in \eqref{E:GN} by a solution $Q_a$ to the elliptic problem
\begin{equation}\label{ell}
-\L Q_a - Q_a + |Q_a|^2Q_a = 0,
\end{equation}
and (ii) for $a>0$, $C_a = C_0$, but equality is never attained in \eqref{E:GN}.

We define the following thresholds:
\begin{equation}\label{thresholds-Q}
\E_a := M(Q_{a\wedge 0})E_{a\wedge 0}(Q_{a\wedge 0}) \qtq{and} \K_a := \|Q_{a\wedge 0}\|_{L_x^2}\|Q_{a\wedge 0}\|_{\dot H_{a\wedge 0}^1}.
\end{equation}
From the Pohozaev identities associated to \eqref{ell} (cf. \eqref{poho} below) and the sharp Gagliardo--Nirenberg inequality \eqref{E:GN}, we compute
\begin{equation}\label{E:threshold}
\E_a = \tfrac{8}{27}C_a^{-2} \qtq{and} \K_a = \tfrac{4}{3} C_a^{-1}.
\end{equation}

Using this notation, the results of  \cite{DHR, HR} may be stated as follows:

\begin{theorem}[Scattering/blowup dichotomy, \cite{DHR, HR}]\label{T:DHR}\text{ }
Let $u_0\in H_x^1(\R^3)$ satisfy $M(u_0)E_0(u_0)<\E_0$.
\begin{itemize}
\item[(i)] If $\|u_0\|_{L_x^2}\| u_0\|_{\dot H_x^1} < \K_0$, then the solution to \eqref{nls0} with initial data $u_0$ is global and scatters.
\item[(ii)] If $\|u_0\|_{L_x^2}\| u_0\|_{\dot H_x^1} > \K_0$ and $u_0$ is radial or $xu_0\in L_x^2(\R^3)$, then the solution to \eqref{nls0} with initial data $u_0$ blows up in finite time in both time directions.
\end{itemize}

Furthermore, if $\psi\in H_x^1(\R^3)$ satisfies $\frac12\|\psi\|_{L_x^2}^2\|\psi\|_{\dot H_x^1}^2< \E_0$, then there exists a global solution to \eqref{nls0} that scatters to $\psi$ forward in time.  The analogous statement holds backward in time.
\end{theorem}

Our main result is the following.

\begin{theorem}[Scattering/blowup dichotomy]\label{T:main} Fix $a>-\frac14$.
Let $u_0\in H_x^1(\R^3)$ satisfy $M(u_0)E_a(u_0)<\E_a$.
\begin{itemize}
\item[(i)] If $\|u_0\|_{L_x^2} \|u_0\|_{\dot H_a^1} < \K_a$, then the solution to \eqref{nls} with initial data $u_0$ is global and scatters.
\item[(ii)] If $\|u_0\|_{L_x^2}\|u_0\|_{\dot H_a^1} > \K_a$ and $u_0$ is radial or $x u_0\in L_x^2(\R^3)$, then the solution to \eqref{nls} with initial data $u_0$ blows up in finite time in both time directions.
\end{itemize}
\end{theorem}

As we will see in Section~\ref{S:var}, the overarching hypothesis $M(u_0)E_a(u_0)<\E_a$ precludes the possibility that $\|u_0\|_{L_x^2} \|u_0\|_{\dot H_a^1} = \K_a$.

For $a\leq 0$, the functions $Q_a$ introduced above provide examples of solutions with $M(u_0)E_a(u_0)=\E_a$ that neither scatter nor blow up; thus, the overarching assumption provides the correct threshold for such a simple dichotomy of behaviors.   A complete description of possible threshold behaviors in the case $a=0$ was given in \cite{DR}.

When $a>0$, there are no optimizers for the Gagliardo--Nirenberg inequality (cf. Theorem~\ref{T:GN} below) and so no soliton solutions with $M(u_0)E_a(u_0)=\E_a$.
Nonetheless, this condition does mark the threshold for uniform space-time bounds.  Specifically, the proof of Theorem~\ref{T:main}(i) will show that such solutions obey
\[
\int_\R \int_{\R^3} |u(t,x)|^5 \,dx\,dt \leq C\bigl( \E_a -M(u_0)E_a(u_0)\bigr)
\]
for some function $C:(0,\E_a)\to(0,\infty)$; however, we can show that the constant here necessarily diverges as one approaches the threshold: 

\begin{theorem}[Failure of uniform space-time bounds at the threshold]\label{T:threshold} Let $a>0$.  There exists a sequence of global solutions $u_n:\R\times\R^3\to\C$ such that
\[
M(u_n)E_a(u_n)\nearrow \E_a\qtq{and} \|u_n(0)\|_{L_x^2} \|u_n(0)\|_{\dot H_a^1} \nearrow \K_a,
\]
with
\[
\lim_{n\to\infty} \|u_n\|_{L_{t,x}^5(\R\times\R^3)} = \infty.
\]
\end{theorem}

The non-existence of optimizers to the Gagliardo--Nirenberg inequality for $a>0$ is a consequence of the failure of compactness due to translations.  If we restrict attention to the Gagliardo--Nirenberg inequality for radial functions, however, the compactness is restored.  To make this precise, we define
\begin{equation}
C_{a,\text{rad}} := \sup\bigl\{ \|f\|_{L_x^4}^4 \div \bigl[\|f\|_{L_x^2} \|f\|_{\dot H_a^1}^3\bigr]:f\in H_a^1\backslash\{0\},\ f\text{ radial}\bigl\}.
\label{E:Carad}
\end{equation}
The proof of Theorem~\ref{T:GN} shows that $C_a = C_{a,\text{rad}}$ for $a\leq 0$.  For $a>0$, we will show that the sharp constant $C_{a,\text{rad}}$ is attained by a radial solution $Q_{a,\text{rad}}$ to \eqref{ell}. However, as the constant $C_a$ defined in \eqref{E:Ca} is \emph{not} attained (cf. Theorem~\ref{T:GN} below), we must have that $C_{a,\text{rad}}<C_a$.

For $a>0$, we define the thresholds
\[
\E_{a,\text{rad}} = \tfrac{8}{27} C_{a,\text{rad}}^{-2} \qtq{and} \K_{a,\text{rad}} = \tfrac43 C_{a,\text{rad}}^{-1},
\]
which are related to $Q_{a,\text{rad}}$ as before.  The functions $Q_{a,\text{rad}}$ provide examples of non-scattering solutions at the radial threshold via $u(t,x) = e^{it}Q_{a,\text{rad}}(x)$, and we obtain the following threshold result for the class of radial solutions:

\begin{theorem}[Radial scattering/blowup dichotomy]\label{T:radial} Fix $a>0$.

Let $u_0\in H_x^1(\R^3)$ be radial and satisfy $M(u_0)E_a(u_0)<\E_{a,\text{rad}}$.
\begin{itemize}
\item[(i)] If $\|u_0\|_{L_x^2} \|u_0\|_{\dot H_a^1} <\K_{a,\text{rad}}$, then the solution to \eqref{nls} with initial data $u_0$ is global and scatters.
\item[(ii)] If $\|u_0\|_{L_x^2} \|u_0\|_{\dot H_a^1} > \K_{a,\text{rad}}$, then the solution to \eqref{nls} with initial data $u_0$ blows up in finite time.
\end{itemize}
\end{theorem}

In particular, as $C_{a,\text{rad}}<C_a$ for $a>0$, the class of radial solutions enjoys strictly larger thresholds for scattering in this case.  Evidently, for $a>0$, there are no radial functions $\phi$ satisfying $M(\phi)E_a(\phi)<\E_a$ and $\K_a < \|\phi\|_{L_x^2} \|\phi\|_{\dot H_a^1} < \K_{a,\text{rad}}$, for in this case, Theorem~\ref{T:main} and Theorem~\ref{T:radial} would contradict one another.

The proof of Theorem~\ref{T:main} comprises most of the paper.  The sharp Gagliardo--Nirenberg inequality \eqref{E:GN}, which we discuss in Section~\ref{S:var}, plays an important role. Virial-type identities also feature heavily in the proof; cf. Section~\ref{S:virial} below.

Part (ii) of Theorem~\ref{T:main} (blowup above the threshold) follows along fairly standard lines (cf. \cite{DHR, Glassey, HR, OT}).  In particular, we combine virial identities with the coercivity stemming from the sharp Gagliardo--Nirenberg inequality (cf. Proposition~\ref{P:coercive}).  We carry out the details in Section~\ref{S:blowup}.

For part (i) of Theorem~\ref{T:main} (scattering below the threshold), we take the concentration compactness approach to induction on energy and argue by contradiction (see \cite{Bourg, CKSTT, KM} for some of the pioneering results).  Supposing the theorem were false, we deduce the existence of a threshold $\E_c \in (0,\E_a)$, which is strictly smaller than the one appearing in Theorem~\ref{T:main}.  We then construct a blowup solution living at this threshold and show that its orbit must be pre-compact in $H_x^1(\R^3)$; this is the content of Theorem~\ref{T:exist}.  Combining this compactness with the virial identity, the sharp Gagliardo--Nirenberg inequality, and the fact that $\E_c<\E_a$, we then deduce a contradiction (see Theorem~\ref{T:not-exist}).

The main component in the proof of Theorem~\ref{T:main}(i) lies in the construction of a minimal blowup solution (Theorem~\ref{T:exist}).  The general strategy is well-established: Firstly, we prove a linear profile decomposition adapted to the $H_x^1\to L_{t,x}^5$ Strichartz inequality for $e^{-it\L}$ (Proposition~\ref{P:LPD}).  Secondly, we prove a Palais--Smale condition for optimizing sequences of initial data (Proposition~\ref{P:PS}).  We refer the reader to \cite{KV-clay, Oberwolfach} for an introduction to these techniques.

The proof of Proposition~\ref{P:LPD} also follows a familiar path: we prove an inverse Strichartz inequality (Proposition~\ref{P:IS}) in order to extract the profiles, each with its own scaling and spatial and temporal translation parameters.  The failure of space-translation symmetry introduces a few wrinkles, related to the convergence of certain linear operators that appear in the argument.  These issues were addressed already in \cite{KMVZZ2}, which treated the energy-critical NLS with inverse-square potential. We import the results we need in Section~\ref{S:coo}.

The main new challenge related to the presence of the potential in \eqref{nls} appears in the proof of the Palais--Smale condition, Proposition~\ref{P:PS}.  The key step is to establish a `nonlinear profile decomposition'; more precisely, given a sequence of optimizing initial data $u_n(0)$, we show that the corresponding solutions can be written approximately as the sum of the nonlinear evolutions of the profiles appearing in the linear profile decomposition for $u_n(0)$.  One expects that the profiles living far from the spatial origin will not be strongly affected by the potential and hence these profiles should be modeled by solutions to \eqref{nls0}, rather than \eqref{nls}.  We make this heuristic precise in Theorem~\ref{T:embedding}. In particular, using the result of \cite{DHR, HR} as a black box and invoking the stability theory for \eqref{nls} (cf. Theorem~\ref{T:stab} below), we construct solutions to \eqref{nls} associated to profiles living far from the origin. To invoke the results of \cite{DHR, HR} requires that the thresholds for \eqref{nls} do not exceed those for \eqref{nls0}; we prove this in Corollary~\ref{C:thresholds} below.

The proofs of Theorems~\ref{T:threshold}~and~\ref{T:radial}, which comprise Section~\ref{S:threshold}, are comparatively quick.  For Theorem~\ref{T:threshold}, we argue as in \cite{KVZ}, choosing a sequence of translates of the ground state for \eqref{nls0} as initial data and appealing to the stability result, Theorem~\ref{T:stab}. For Theorem~\ref{T:radial}, the main arguments are the same as those used to prove Theorem~\ref{T:main}.  Hence we provide only a sketch, pointing out the few places where the radial assumption comes into play.

Several previous works have treated dispersive equations with broken symmetries; see, for example, \cite{IP1, IP2, IPS, Jao1, Jao2, KKSV, KMVZZ2, KOPV, KSV, KVZ0, KVZ, PTW}. As in these works, we see that in order to treat a problem with broken symmetries, one needs a good understanding of the limiting problem in which the symmetries are restored.  In our case, the limiting problem is \eqref{nls0}, which was studied in \cite{DHR, HR}.  See also \cite{KVZ}, which studied the focusing, cubic NLS in the exterior of a convex obstacle, for which \eqref{nls0} is again the relevant limiting problem, as well as \cite{KMVZZ2}, which also considered NLS with an inverse-square potential.

The rest of the paper is organized as follows:  In Section~\ref{S:prelim}, we collect some useful lemmas, including some harmonic analysis tools from \cite{KMVZZ1} related to the operator $\L$.  We also establish some local theory and stability results, namely, Theorem~\ref{T:LWP} and Theorem~\ref{T:stab}.  In Section~\ref{S:var}, we carry out the variational analysis for the sharp Gagliardo--Nirenberg inequality.  In Section~\ref{S:blowup}, we prove part (ii) of Theorem~\ref{T:main}, establishing blowup above the threshold.

Section~\ref{S:cc} contains the proof of the linear profile decomposition, Proposition~\ref{P:LPD}.  Section~\ref{S:embed} contains the proof of Theorem~\ref{T:embedding}, in which we construct nonlinear solutions associated to profiles living far from the origin.  In Section~\ref{S:exist}, we prove Theorem~\ref{T:exist}, which asserts that the failure of Theorem~\ref{T:main} implies the existence of minimal blowup solutions.  In Section~\ref{S:not-exist}, we show that such solutions cannot exist, thus completing the proof of Theorem~\ref{T:main}.

Finally, in Section~\ref{S:threshold}, we prove Theorems~\ref{T:threshold}~and~\ref{T:radial}.

\subsection*{Acknowledgements} R.K. was supported by the NSF grant DMS-1265868.  J.M. was supported by the NSF Postdoctoral Fellowship DMS-1400706.  M.V. was supported by the NSF grant DMS-1500707.  J.Z. was partly supported by the European Research Council, ERC-2012-ADG, project number 320845: Semi-Classical Analysis of Partial Differential Equations.  The work on this project was supported in part by NSF grant DMS-1440140, while the authors were in residence at the Mathematical Sciences Research Institute in Berkeley, California, during the Fall 2015 semester.

\section{Preliminaries}\label{S:prelim}

We begin by introducing some notation. For non-negative quantities $X$ and $Y$, we write $X\lesssim Y $ or $X=O(Y)$ when $X\leq CY$ for some $C>0$.  We write $A\wedge B = \min\{A,B\}$,  $A\vee B =\max\{A, B\}$,  and $\langle x\rangle =\sqrt{1+|x|^2}.$

For $1< r < \infty$, we write $\dot H^{s,r}_a(\R^3)$ and $H^{s,r}_a(\R^3)$ for the homogeneous and inhomogeneous Sobolev spaces associated with $\mathcal{L}_a$, respectively, which have norms
\[
\|f\|_{\dot H^{s,r}_a(\R^3)}= \|(\mathcal{L}_a)^{\frac{s}2} f\|_{L_x^r(\R^3)} \qtq{and} \|f\|_{H^{s,r}_a(\R^3)}= \|(1+ \mathcal{L}_a)^\frac{s}2 f\|_{L_x^r(\R^3)}.
\]
We abbreviate $\dot H^{s}_a(\R^3)=\dot H^{s,2}_a(\R^3)$ and $H^{s}_a(\R^3)=H^{s,2}_a(\R^3)$.  We use the notation
\[
\|u\|_{L_t^q L_x^r(I\times\R^3)} = \biggl(\int_I \biggl( \int_{\R^3} |u(t,x)|^r \,dx \biggr)^{\frac{q}{r}}\,dt\biggr)^{\frac{1}{q}},
\]
with the usual modifications if $q$ or $r$ equals $\infty$.

\subsection{Harmonic analysis for $\L$} In this section, we collect some harmonic analysis tools adapted to the operator $\L$.  The primary reference for this section is \cite{KMVZZ1}.

Some of the results admit generalizations to dimensions $d\geq 3$.   By the sharp Hardy inequality, the operator $\L$ is positive precisely for $a\geq -(\frac{d-2}2)^2$.  Many results have clearer formulations when written in terms of the parameter
\[
\sigma:=\tfrac{d-2}2-\bigr[\bigl(\tfrac{d-2}2\bigr)^2+a\bigr]^{\frac12}.
\]

Estimates on the heat kernel associated to the operator $\mathcal{L}_a$ were found by Liskevich--Sobol \cite{LS} and Milman--Semenov \cite{MS}.

\begin{lemma}[Heat kernel bounds, \cite{LS, MS}] \label{L:kernel}  Let $d\geq 3$ and $a\geq -(\tfrac{d-2}{2})^2$. There exist positive constants $C_1,C_2$ and $c_1,c_2$ such that for any $t>0$ and any $x,y\in\R^d\backslash\{0\}$,
\[
C_1(1\vee\tfrac{\sqrt{t}}{|x|})^\sigma(1\vee\tfrac{\sqrt{t}}{|y|})^\sigma t^{-\frac{d}{2}} e^{-\frac{|x-y|^2}{c_1t}} \leq e^{-t\L}(x,y) \leq
C_2(1\vee\tfrac{\sqrt{t}}{|x|})^\sigma(1\vee\tfrac{\sqrt{t}}{|y|})^\sigma t^{-\frac{d}{2}} e^{-\frac{|x-y|^2}{c_2t}}.
\]
\end{lemma}

These estimates formed the starting point of the analysis in \cite{KMVZZ1}, which developed a number of basic harmonic analysis tools that will be of use in this paper.  Foremost among these is the following, which tells us when Sobolev spaces defined through powers of $\L$ coincide with the traditional Sobolev spaces.

\begin{lemma}[Equivalence of Sobolev spaces, \cite{KMVZZ1}]\label{pro:equivsobolev} Let $d\geq 3$, $a\geq -(\frac{d-2}{2})^2$, and $0<s<2$. If $1<p<\infty$ satisfies $\frac{s+\sigma}{d}<\frac{1}{p}< \min\{1,\frac{d-\sigma}{d}\}$, then
\[
\||\nabla|^s f \|_{L_x^p}\lesssim_{d,p,s} \|(\L)^{\frac{s}{2}} f\|_{L_x^p}\qtq{for all}f\in C_c^\infty(\R^d\backslash\{0\}).
\]
If $\max\{\frac{s}{d},\frac{\sigma}{d}\}<\frac{1}{p}<\min\{1,\frac{d-\sigma}{d}\}$, then
\[
\|(\L)^{\frac{s}{2}} f\|_{L_x^p}\lesssim_{d,p,s} \||\nabla|^s f\|_{L_x^p} \qtq{for all} f\in C_c^\infty(\R^d\backslash\{0\}).
\]
\end{lemma}

\begin{remark} We consider the full range $a>-\frac14$, which puts some restrictions on the spaces we may use.  This accounts for some of the more peculiar exponents appearing in Section~\ref{S:embed} and Section~\ref{S:exist}.
\end{remark}

We will make use of the following fractional calculus estimates due to Christ and Weinstein \cite{CW}.  Combining these estimates with Lemma~\ref{pro:equivsobolev}, we can deduce analogous statements for the operator $\L$ (for restricted sets of exponents).

\begin{lemma}[Fractional calculus]\text{ }
\begin{itemize}
\item[(i)] Let $s\geq 0$ and $1<r,r_j,q_j<\infty$ satisfy $\tfrac{1}{r}=\tfrac{1}{r_j}+\tfrac{1}{q_j}$ for $j=1,2$. Then
\[
\| |\nabla|^s(fg) \|_{L_x^r} \lesssim \|f\|_{L_x^{r_1}} \||\nabla|^s g\|_{L_x^{q_1}} + \| |\nabla|^s f\|_{L_x^{r_2}} \| g\|_{L_x^{q_2}}.
\]
\item[(ii)] Let $G\in C^1(\C)$ and $s\in (0,1]$, and let $1<r_1\leq \infty$  and $1<r,r_2<\infty$ satisfy $\tfrac{1}{r}=\tfrac{1}{r_1}+\tfrac{1}{r_2}$. Then
\[
\| |\nabla|^s G(u)\|_{L_x^r} \lesssim \|G'(u)\|_{L_x^{r_1}} \|u\|_{L_x^{r_2}}.
\]
\end{itemize}
\end{lemma}

We will make use of Littlewood--Paley projections defined via the heat kernel:
\begin{align*}
P_N^a:=e^{-\mathcal L_a/N^2}-e^{-4\mathcal L_a/N^2}\qtq{for} N \in 2^{\mathbb{Z}}.
\end{align*}

To state the following results, we define
\[
q_0 := \begin{cases} \infty & \text{if }a\geq 0, \\ \tfrac{d}{\sigma} & \text{if }-(\frac{d-2}{2})^2\leq a < 0, \end{cases}
\]
and we write $q_0'$ for the dual exponent to $q_0$.  We begin with several lemmas proved in \cite{KMVZZ1} on the basis of the (Mikhlin-type) multiplier theorem proved therein for functions of the operator $\L$.

\begin{lemma}[Expansion of the identity, \cite{KMVZZ1}] Let $q_0' < r < q_0$. Then
\[
f= \sum_{N\in 2^{\mathbb{Z}}} P_N^a f\qtq{as elements of} L_x^r.
\]
\end{lemma}

\begin{lemma}[Bernstein estimates, \cite{KMVZZ1}]\label{L:Bernie} Let $q_0'<q\leq r<q_0$.  Then
\begin{itemize}
\item[(i)] The operators $P^a_{N}$ are bounded on $L_x^r$.
\item[(ii)] The operators $P^a_{N}$ map $L_x^q$ to $L_x^r$, with norm $O(N^{\frac dq-\frac dr})$.
\item[(iii)] For any $s\in \R$,
\[
N^s\|P^a_N f\|_{L_x^r}  \sim \bigl\|(\L)^{\frac s2}P^a_N f\bigr\|_{L_x^r}.
\]
\end{itemize}
\end{lemma}

\begin{lemma}[Square function estimate, \cite{KMVZZ1}]\label{T:sq}
Let $0\leq s<2$ and $q_0'<r<q_0$. Then
\begin{align*}
\biggl\|\biggl(\sum_{N\in2^\mathbb{Z}} N^{2s}| P^a_N f|^2\biggr)^{\!\!\frac 12}\biggr\|_{L_x^r} \sim \|(\L)^{\frac s2}f\|_{L_x^r}.
\end{align*}
\end{lemma}

The following refined Fatou lemma due to Brezis and Lieb \cite{BL} will be of use in Section~\ref{S:cc}.

\begin{lemma}[Refined Fatou, \cite{BL}]\label{L:RF} Let $1\leq r<\infty$ and let $\{f_n\}$ be a bounded sequence in $L_x^r$. If $f_n\to f$ almost everywhere, then
\[
\int \bigl| |f_n|^r - |f_n-f|^r - |f|^r \bigr| \,dx \to 0.
\]
\end{lemma}

Strichartz estimates for the propagator $e^{-it\L}$ in $\R^3$ were proved by Burq, Planchon, Stalker, and Tahvildar-Zadeh in \cite{BPSTZ}.  Combining these with the Christ--Kiselev Lemma \cite{CK}, we obtain the following Strichartz estimates:

\begin{proposition}[Strichartz, \cite{BPSTZ}] Fix $a>-\frac14$. The solution $u$ to $(i\partial_t-\L)u = F$ on an interval $I\ni t_0$ obeys
\[
\|u\|_{L_t^q L_x^r(I\times\R^3)} \lesssim \|u(t_0)\|_{L_x^2(\R^3)} + \|F\|_{L_t^{\tilde q'} L_x^{\tilde r'}(I\times\R^3)}
\]
for any $2\leq q,\tilde q\leq\infty$ with $\frac{2}{q}+\frac{3}{r}=\frac{2}{\tilde q}+\frac{3}{\tilde r}= \frac32$ and $(q,\tilde q)\neq (2,2)$. \end{proposition}

Note the loss of the double endpoint in the theorem above.  For $a\geq 0$, the double endpoint estimate can be recovered via the argument in \cite{KeelTao}, since $L^1_x\to L_x^\infty$ dispersive estimates hold in this case (see \cite[Theorem~1.11]{Fanelli}).  For sufficiently small $a<0$, one can also recover the double endpoint estimate via the argument in \cite{BPSTZ}.

Throughout the paper, it will be convenient to use the following notation:
\[
S^s_a(I) = L_t^2 H_a^{s,6}\cap L_t^\infty H_a^s(I\times\R^3) \qtq{and}
\dot S^s_a(I) = L_t^2 \dot H_a^{s,6}\cap L_t^\infty \dot H_a^s(I\times\R^3).
\]

\subsection{Convergence of operators}\label{S:coo}

In this section, we import a convergence result from \cite{KMVZZ2}; its consequences will be useful in Sections~\ref{S:cc}~and~\ref{S:embed}.

We begin with a definition, the need for which reflects the fact that $\L$ does not commute with translations.

\begin{definition}\label{D:ops} Suppose $\{x_n\}\subset\R^3$. We define
\[
\L^n = -\Delta + \tfrac{a}{|x+x_n|^2} \qtq{and}
\L^\infty = \begin{cases} -\Delta + \tfrac{a}{|x+x_\infty|^2} & \text{if}\quad x_n\to x_\infty\in \R^3, \\
    -\Delta & \text{if}\quad |x_n|\to\infty.\end{cases}
\]
In particular, $\L[\phi(x-x_n)] = [\L^n \phi](x-x_n).$
\end{definition}

 The operators $\L^\infty$ appear as limits of the operators $\L^n$, as in the following lemma.

\begin{lemma}[Convergence of operators, \cite{KMVZZ2}] \label{L:coo} Let $a>-\frac14$. Suppose $\tau_n\to \tau_\infty\in\R$ and $\{x_n\}\subset\R^3$ satisfies $x_n\to x_\infty\in\R^3$ or $|x_n|\to\infty$. Then,
\begin{align}
\label{coo1}
&\lim_{n\to\infty} \|\L^n \psi - \L^\infty \psi \|_{\dot H_x^{-1}} = 0\qtq{for all} \psi\in \dot H_x^1, \\
\label{coo3}
&\lim_{n\to\infty} \|\bigl( e^{-i\tau_n\L^n}-e^{-i\tau_\infty\L^\infty}\bigr)\psi\|_{\dot H_x^{-1}} =0 \qtq{for all} \psi\in \dot H_x^{-1}, \\
\label{coo4}
&\lim_{n\to\infty} \| \bigl[({\L^n})^{\frac12} -(\L^\infty)^{\frac12}\bigr]\psi\|_{L_x^2} = 0 \qtq{for all} \psi\in\dot H_x^1.
\end{align}
Furthermore, for any $2< q\leq \infty$ and $\frac{2}{q}+\frac{3}{r}=\frac{3}{2}$,
\begin{equation}
\label{coo5}
\lim_{n\to\infty} \|\bigl(e^{-it\L^n}-e^{-it\L^\infty}\bigr)\psi\|_{L_t^q L_x^r(\R\times\R^3)} = 0 \qtq{for all} \psi\in L_x^2.
\end{equation}
Finally, if $x_\infty\neq 0$, then for any $t>0$,
\begin{align}
\label{coo2}
\lim_{n\to\infty} \| [e^{-t\L^n}-e^{-t\L^\infty}]\delta_0 \|_{\dot H_x^{-1}} = 0.
\end{align}
\end{lemma}

\begin{remark} In \cite{KMVZZ2}, \eqref{coo2} appears with $t=1$; the case of general $t>0$ follows by scaling.
\end{remark}

In \cite[Corollary~3.4]{KMVZZ2}, the authors use \eqref{coo4} and \eqref{coo5} to prove
\[
\lim_{n\to\infty} \|e^{it_n\L^n}\psi\|_{L_x^6}=0\qtq{for}\psi\in \dot H_x^1.
\]
Interpolating this with $L_x^2$-boundedness yields the following corollary, which will be used in the proof of Proposition~\ref{P:IS}:

\begin{corollary}\label{L4-to-zero} For $\{x_n\}\subset\R^3$, $t_n\to\pm\infty$, and $\psi \in H_x^1$, we have
\[
\lim_{n\to\infty} \|e^{it_n\L^n}\psi \|_{L_x^4} = 0.
\]
\end{corollary}

We record one more corollary, which will be used in the proof of Theorem~\ref{T:embedding}.

\begin{corollary}\label{S-to-zero} Let $a>-\frac14$. Suppose $x_n\to x_\infty\in\R^3$ or $|x_n|\to\infty$.  Then
\[
\lim_{n\to\infty} \|(e^{-it\L^n} - e^{-it\L^\infty})\psi\|_{L_{t,x}^5(\R\times\R^3)} = 0 \qtq{for all}\psi\in H_x^1.
\]
\end{corollary}

\begin{proof} Fix $\psi\in H^1_x$. By Sobolev embedding, it suffices to show
\[
\lim_{n\to\infty}\||\nabla|^{\frac12}(e^{-it\L^n}-e^{-it\L^\infty})\psi\|_{L_t^5 L_x^{\frac{30}{11}}}=0.
\]
To this end, we first use \eqref{coo5} to see that
\[
\lim_{n\to\infty}\| (e^{-it\L^n}-e^{-it\L^\infty})\psi\|_{L_t^5 L_x^{\frac{30}{11}}} = 0.
\]
On the other hand, by equivalence of Sobolev spaces and Strichartz, we have
\[
\sup_n \||\nabla|^{\frac35}(e^{-it\L^n} -e^{-it\L^\infty})\psi\|_{L_t^5 L_x^{\frac{30}{11}}}\lesssim \| |\nabla|^{\frac35}\psi\|_{L_x^2}\lesssim 1.
\]
The result now follows by interpolation.\end{proof}

\subsection{Local theory and stability}
In this section, we establish the local well-posedness theory for \eqref{nls} via the usual Strichartz methodology.  Equivalence of Sobolev norms (cf. Lemma~\ref{pro:equivsobolev}) places severe restrictions on which spaces can be used; nonetheless, we have been able to find an argument that works seamlessly for all $a>-\tfrac14$.  We present the details below.  By contrast, local well-posedness for the energy-critical NLS with inverse square potential is currently open for $a$ close to  $a>-\tfrac14$; see \cite{KMVZZ2}.

\begin{theorem}[Local well-posedness]\label{T:LWP} Fix $a>-\frac14$, $u_0\in H_x^1(\R^3)$, and $t_0\in\R$.  Then the following hold:

\begin{itemize}
\item[(i)] There exist $T=T(\|u_0\|_{H_a^1})>0$ and a unique solution $u:(t_0-T,t_0+T)\times\R^3\to\C$ to \eqref{nls} with $u(t_0)=u_0$. In particular, if $u$ remains uniformly bounded in $H_a^1$ throughout its lifespan, then $u$ extends to a global  solution.
\item[(ii)]  There exists $\eta_0>0$ such that if
\[
\|e^{-i(t-t_0)\L}u_0\|_{L_{t,x}^5((t_0,\infty)\times\R^3)} < \eta\qtq{for some} 0 <\eta<\eta_0,
\]
then the solution $u$ to \eqref{nls} with data $u(t_0)=u_0$ is forward-global and satisfies
\[
\|u\|_{L_{t,x}^5((t_0,\infty)\times\R^3)} \lesssim \eta.
\]
The analogous statement holds backward in time (as well as on all of $\R$).
\item[(iii)] For any $\psi\in H_a^1$, there exist $T>0$ and a solution $u:(T,\infty)\times\R^3\to\C$ to \eqref{nls} such that
\[
\lim_{t\to\infty} \|u(t) - e^{-it\L}\psi\|_{H_a^1} = 0.
\]
The analogous statement holds backward in time.
\end{itemize}
\end{theorem}

\begin{proof} We first prove (i).  By time-translation invariance, we may assume $t_0=0$.

Fix $T>0$ and set $A=\|u_0\|_{H_a^1}$.  Throughout the proof of (i), all space-time norms will be taken over $(-T,T)\times\R^3$. It suffices to show that for $T$ sufficiently small, the map $\Phi$ defined by
\[
[\Phi u](t) = e^{-it\L}u_0 + i\int_0^t e^{-i(t-s)\L}\bigl(|u(s)|^2 u(s)\bigr)\,ds
\]
is a contraction on the space
\[
B_T = \{u\in C_t H_a^1 \cap L_t^5 H_a^{1,\frac{30}{11}}: \|u\|_{L^\infty_t H_a^1} \leq CA,\quad \|u\|_{L_t^5 H_a^{1,\frac{30}{11}}}\leq CA\},
\]
which is complete with respect to the metric
\[
d(u,v) = \| u - v \|_{L_t^5 L_x^{\frac{30}{11}}}.
\]
Here, $C$ encodes the various constants appearing in Strichartz estimates and Sobolev embedding.

Let $u\in B_T$. By Sobolev embedding and equivalence of Sobolev spaces,
\[
\|u\|_{L_t^5 L_x^6} \lesssim \| u\|_{L_t^5 \dot H_a^{\frac35,\frac{30}{11}}}  \lesssim A.
\]
Thus, by Strichartz, Sobolev embedding, and equivalence of Sobolev spaces,
\begin{align*}
\|\Phi u\|_{L^\infty_t H_a^1}\vee \|\Phi u\|_{L_t^5 H_a^{1,\frac{30}{11}}} & \lesssim \|u_0\|_{H_a^1} + \| |u|^2 u\|_{L_t^2 H_x^{1,\frac65}} \\
& \lesssim A + |T|^{\frac{1}{10}} \| u\|_{L_t^5 L_x^6}^2 \|u\|_{L^\infty_t H_a^1} \\
& \lesssim A + |T|^{\frac{1}{10}} (CA)^3.
\end{align*}
Thus $\Phi:B_T\to B_T$, provided $C$ is chosen sufficiently large and $T=T(\|u_0\|_{H_a^1})$ sufficiently small. Similarly, for $u,v\in B_T$,
\begin{align*}
\|\Phi u-\Phi v\|_{L_t^5 L_x^{\frac{30}{11}}}  \lesssim \| |u|^2u - |v|^2 v\|_{L_{t,x}^{\frac{10}{7}}} & \lesssim |T|^{\frac{1}{10}}\|u-v\|_{L_t^5 L_x^{\frac{30}{11}}} \bigl( \| u\|_{L_t^5 L_x^6}^2 + \|v\|_{L_t^5 L_x^6}^2 \bigr) \\
& \lesssim |T|^{\frac{1}{10}} A^2 \|u-v\|_{L_t^5 L_x^{\frac{30}{11}}},
\end{align*}
so that $\Phi$ is a contraction on $B_T$, provided $T=T(\|u_0\|_{H_a^1})$ is sufficiently small.  This completes the proof of (i).

We turn to (ii). Once again, we take $t_0=0$, set $A=\|u_0\|_{H_a^1}$, and define $\Phi$ as above. This time, we run a contraction mapping argument on the complete metric space
\begin{align*}
& B_{\eta} = \{u\in C_t H_a^{\frac12} \cap L_t^5 H_a^{\frac12,\frac{30}{11}}: \|u\|_{L^\infty_t H_a^{\frac12}}\leq CA,\,  \|u\|_{L_t^5 H_a^{\frac12,\frac{30}{11}}}\leq CA, \, \|u\|_{L_{t,x}^5}\leq 2\eta\}, \\
& d(u,v) = \|u-v\|_{L_t^{5}L_x^{\frac{30}{11}}},
\end{align*}
where space-time norms are over $(0,\infty)\times\R^3$. Given $u\in B_\eta$, we can estimate as above to find that
\begin{align*}
\| \Phi u\|_{L^\infty_t H_a^{\frac12}} + \|\Phi u\|_{L_t^5 H_a^{\frac12,\frac{30}{11}}}  \lesssim A + \| |u|^2 u \|_{L_t^{\frac53} H_x^{\frac12,\frac{30}{23}}}  & \lesssim A + \|u\|_{L_{t,x}^5}^2 \|u\|_{L_t^5 H_a^{\frac12,\frac{30}{11}}} \\
& \lesssim A + CA(2\eta)^2  \leq CA
\end{align*}
for $C$ large $\eta$ sufficiently small. Similarly, by Sobolev embedding,
\[
\|\Phi u\|_{L_{t,x}^5} \leq \|e^{-it\L}u_0\|_{L_{t,x}^5} + \|u\|_{L_{t,x}^5}^2 \|u\|_{L_t^5 H_a^{\frac12,\frac{30}{11}}}
    \lesssim \eta + C\eta^2A\leq C\eta
\]
for $\eta$ sufficiently small. Thus $\Phi:B_\eta\to B_\eta$.  Similarly, for any $u,v\in B_\eta$,
\[
\|\Phi u-\Phi v\|_{L_t^5 L_x^{\frac{30}{11}}} \lesssim \|u-v\|_{L_t^5 L_x^{\frac{30}{11}}}\bigl(\|u\|_{L_{t,x}^5}^2+\|v\|_{L_{t,x}^5}^2\bigr)
\lesssim \eta^2 \|u-v\|_{L_t^5 L_x^{\frac{30}{11}}},
\]
so that $\Phi$ is a contraction for $\eta$ sufficiently small.  This yields a global solution
\[
u\in C_t H_a^{\frac12}\cap L_t^5 H_a^{\frac12,\frac{30}{11}}
\]
obeying the desired $L_{t,x}^5$ bounds.  To upgrade $u$ to a solution in the sense of Definition~\ref{D:solution}, recall that $u_0\in H_a^1$ and see Remark~\ref{remark1} below.

For (iii), we instead seek a fixed point of the equation
\[
u(t) = e^{-it\L} \psi - i\int_t^\infty e^{-i(t-s)\L} \bigl(|u(s)|^2 u(s)\bigr) \,ds
\]
on $(T,\infty)\times\R^3$ for some large $T$.  For this, we can use the same spaces used to prove (ii). The key observation is the following: as $\psi \in H_a^1$, Sobolev embedding, Strichartz, and monotone convergence imply that for any $\eta>0$, there exists $T$ large enough that
\[
\|e^{-it\L}\psi\|_{L_{t,x}^5((T,\infty)\times\R^3)} < \eta.
\]
Thus, we can construct a solution that scatters to $e^{-it\L}\psi$ in the $H_a^{\frac12}$-topology.  Using Remark~\ref{remark1}, we can upgrade $u$ to a solution in the sense of Definition~\ref{D:solution}, as well as prove that scattering holds in $H_a^1$.
\end{proof}

\begin{remark}[Persistence of regularity]\label{remark1} Suppose $u:I\times\R^3\to\C$ is a solution to \eqref{nls} such that
\[
\|u(t_0)\|_{H_x^1} \leq E\qtq{for some $t_0\in I$}\qtq{and}\quad \|u\|_{L_{t,x}^5(I\times\R^3)} \leq L.
\]
Then
\begin{equation}\label{in-fact}
\|u\|_{S_a^1(I)} \lesssim_{E,L} 1,
\end{equation}
which has the following consequences:
\begin{itemize}
\item[(i)] If the $L_{t,x}^5$-norm of a solution remains bounded throughout its lifespan, the solution may be extended globally in time.
\item[(ii)] If the solution belongs to $L_{t,x}^5(\R\times\R^3)$, the solution scatters in $H_a^1$.
\end{itemize}

Indeed, for (i) we need only note that in this case, $u$ remains uniformly bounded in $H_a^1$.  For (ii), one can use Strichartz and estimate as in \eqref{remark-est} below  to find that
\[
\|e^{it\L}u(t) - e^{is\L}u(s)\|_{H_a^1} \lesssim \|u\|^2_{L_t^4 \dot H_a^{\frac12,3}((s,t)\times\R^3)} \|u\|_{L_t^\infty H_a^1},
\]
which implies that $\{e^{it\L}u(t)\}_{t\in\R}$ is Cauchy in $H^1_a$ as $t\to\pm\infty$.

\begin{proof}[Proof of \eqref{in-fact}]  In the following, we take all space-time norms over $I\times\R^3$.

Using Strichartz and the equivalence of Sobolev spaces, we first estimate
\begin{align}\nonumber
\|u\|_{L_t^\infty \dot H_a^{\frac12}} + \| u\|_{L_t^4 \dot H_a^{\frac12,3}} + \| u\|_{L_t^5 \dot H_a^{\frac12,\frac{30}{11}}}
&\lesssim \| u(t_0)\|_{\dot H_a^{\frac12}} + \||\nabla|^{\frac12}(|u|^2 u)\|_{L_t^{\frac53}L_x^{\frac{30}{23}}} \\
\label{boot}
& \lesssim_E 1+\|u\|_{L_{t,x}^5}^2 \|u\|_{L_t^5 \dot H_a^{\frac12,\frac{30}{11}}}.
\end{align}
Thus, a standard bootstrap yields $u\in L_t^4 \dot H_a^{\frac12,3}$. Using the estimate
\begin{equation}\label{remark-est}
\| |u|^2 u\|_{L_t^2 \dot H_a^{1,\frac65}} \lesssim \|u\|_{L_t^4 L_x^6}^2 \|\nabla u\|_{L_t^\infty L_x^2}
    \lesssim \|u\|_{L_t^4 \dot H_a^{\frac12,3}}^2 \|\nabla u\|_{L_t^\infty L_x^2},
\end{equation}
one can now deduce that $u\in L_t^q H_a^{1,r}$ for any admissible $(q,r)$ with $q>2$.  To get the endpoint $L_t^2 L_x^6$, we use
\[
\| |u|^2 u\|_{L_t^{\frac{4}{2+\eps}}\dot H_a^{1,\frac{6}{5-\eps}}} \lesssim \|u\|_{L_t^{\frac{8}{2+\eps}}
 L_x^{\frac{12}{2-\eps}}}^2 \|\nabla u\|_{L_t^\infty L_x^2} \lesssim
  \|u\|_{L_t^{\frac{8}{2+\eps}}\dot H_a^{\frac12,\frac{12}{4-\eps}}}^2 \|\nabla u\|_{L_t^\infty L_x^2},
\]
where $\eps=\eps(a)>0$ is chosen small enough that $\dot H_x^{\frac12,\frac{12}{4-\eps}}$ is equivalent to $\dot H_a^{\frac12,\frac{12}{4-\eps}}$. This completes the proof of \eqref{in-fact}.
\end{proof}
\end{remark}

We next record a stability result for \eqref{nls}, which will play an important role in the proofs of Theorems~\ref{T:embedding}~and~\ref{T:exist}.  The proof is standard and relies on the estimates used above.

\begin{theorem}[Stability]\label{T:stab}  Fix $a>-\frac14$. Let $\tilde{v}:I\times\R^3\to\C$ solve
\[
(i\partial_t-\L)\tilde{v} = -|\tilde{v}|^2\tilde{v} + e, \quad v(t_0)= \tilde v_0\in H_x^1(\R^3)
\]
for some $t_0\in I$ and some `error' $e:I\times\R^3\to\C$. Fix $v_0\in H^1_x$ and suppose
\[
\|v_0\|_{H^1_x} + \|\tilde{v}_0\|_{H_x^1} \leq E \qtq{and} \| \tilde{v}\|_{L_{t,x}^5} \leq L
\]
for some $E,L>0$.  There exists $\eps_0=\eps_0(E,L)>0$ such that if $0<\eps<\eps_0$ and
\begin{equation}\label{stab:small}
\|\tilde v_0 - v_0 \|_{\dot H_x^{\frac12}} + \||\nabla|^{\frac12} e\|_{N(I)} < \eps,
\end{equation}
where
\[
N(I):=L_{t,x}^{\frac{10}{7}}(I\times\R^3)+L_t^{\frac53}L_x^{\frac{30}{23}}(I\times\R^3)+ L_t^1 L_x^2(I\times\R^3),
\]
then there exists a solution $v:I\times\R^3\to\C$ to \eqref{nls} with $v(t_0)=v_0$ satisfying
\begin{align}
\label{E:stab}
&\|v-\tilde v\|_{\dot S_a^{\frac12}(I\times\R^3)} \lesssim_{E,L} \eps, \\
\label{E:stab-bound}
&\| v\|_{S_a^1(I\times\R^3)}\lesssim_{E,L} 1.
\end{align}

If, in addition,
\[
\|\tilde v_0 - v_0 \|_{\dot H_x^{\frac35}} + \||\nabla|^{\frac35} e\|_{N(I)} < \eps,
\]
then we may also conclude that
\begin{align}\label{E:stab2}
\| v-\tilde v\|_{\dot S_a^{\frac35}(I\times\R^3)}\lesssim_{E,L}\eps.
\end{align}
\end{theorem}

\subsection{Virial identities}\label{S:virial}  In this section, we recall some standard virial-type identities.  Given a weight $w:\R^3\to\R$ and a solution $u$ to \eqref{nls}, we define
\[
V(t;w) := \int |u(t,x)|^2 w(x)\,dx.
\]
Using \eqref{nls}, one finds
\begin{equation}\label{virial}
\begin{aligned}
& \partial_t V(t;w) = \int 2\Im \bar u \nabla u \cdot \nabla w \,dx, \\
& \partial_{tt}V(t;w) = \int (-\Delta\Delta w)|u|^2 + 4\Re \bar u_j u_k w_{jk} + 4|u|^2 \tfrac{ax}{|x|^4}\cdot \nabla w - |u|^4 \Delta w\,dx.
\end{aligned}
\end{equation}

The standard virial identity makes use of $w(x)=|x|^2$.

\begin{lemma}[Standard virial identity]\label{L:virial0} Let $u$ be a solution to \eqref{nls}.  Then
\[
\partial_{tt} V(t;|x|^2) = 8\Bigl[\|u(t)\|_{\dot H_a^1}^2 - \tfrac34\|u(t)\|_{L_x^4}^4\Bigr].
\]
\end{lemma}

In general, we do not work with solutions for which $V(t;|x|^2)$ is finite. Thus, we need a truncated version of the virial identity (cf. \cite{OT}, for example). For $R>1$, we define $w_R(x) = R^2\phi(\frac{x}{R})$, where $\phi$ is a smooth, non-negative radial function satisfying
\begin{equation}\label{phi-virial}
\phi(x)=\begin{cases} |x|^2 & |x|\leq 1 \\ 9 & |x|>3,\end{cases} \qtq{with} |\nabla \phi|\leq 2|x|,\quad |\partial_{jk}\phi|\leq 2.
\end{equation}

In this case, we use \eqref{virial} to deduce the following:

\begin{lemma}[Truncated virial identity]\label{L:virial} Let $u$ be a solution to \eqref{nls} and let $R>1$. Then
\begin{align}
\nonumber \partial_{tt}& V(t;w_R)\\
\nonumber
  &= 8\Bigl[\|u(t)\|_{\dot H_a^1}^2 - \tfrac34 \|u(t)\|_{L_x^4}^4\Bigr]  \\
\label{virial-sign}
& \quad  + 4\int_{|x|>R} \Re \bar u_j u_k\partial_{jk}[w_R]  + |u|^2 \tfrac{ax}{|x|^4}\cdot \nabla w_R \,dx - 8\int_{|x|>R} |(\L)^{\frac12} u|^2 \,dx \\
\nonumber &\quad  +O\biggl( \int_{|x|\geq R} R^{-2}|u|^2 + |u|^4
\,dx\biggr).
\end{align}
Furthermore, by \eqref{phi-virial}, we have that $\eqref{virial-sign}\leq 0$.
\end{lemma}

\section{Variational analysis}\label{S:var}

In this section, we carry out the variational analysis for the sharp Gagliardo--Nirenberg inequality \eqref{E:GN}, which leads naturally to the thresholds appearing in Theorem~\ref{T:main}.

\begin{theorem}[Sharp Gagliardo--Nirenberg inequality]\label{T:GN} Fix $a>-\frac14$ and define
\[
C_a := \sup\bigl\{ \|f\|_{L_x^4}^4 \div \bigl[\|f\|_{L_x^2} \|f\|_{\dot H_a^1}^3\bigr]:f\in H_a^1\backslash\{0\}\ \bigl\}.
\]
Then $C_a\in(0, \infty)$ and the following hold:
\begin{itemize}
\item[(i)] If $a\leq 0$, then equality in the Gagliardo--Nirenberg inequality \eqref{E:GN} is attained by a function $Q_a\in H_a^1$, which is a non-zero, non-negative, radial solution to the elliptic problem
\begin{equation}
\label{elliptic}
-\L Q_a - Q_a + Q_a^3 = 0.
\end{equation}
\item[(ii)] If $a>0$, then $C_a = C_0$, but equality in \eqref{E:GN} is never attained.
\end{itemize}
\end{theorem}

\begin{proof}   Define the functional
\[
J_a(f) := \frac{\|f\|_{L_x^4}^4}{\|f\|_{L_x^2}\|f\|_{\dot H^1_a}^3},
\qtq{so that} C_a = \sup\{J_a(f):f\in H_a^1\backslash\{0\}\}.
\]
Note that the standard Gagliardo--Nirenberg inequality and the equivalence of Sobolev spaces imply  $0<C_a<\infty$.

We first prove existence of optimizers in the case $a\leq 0$ by mimicking the well-known proof for $a=0$.  Take $\{f_n\}\subset H_a^1\backslash\{0\}$ such that $J_a(f_n)\nearrow C_a$ and let $f_n^*$ denote the decreasing spherically symmetric rearrangement of $f_n$.  For $a\leq 0$, we have $J_a(f_n)\leq J_a(f_n^*)$; this relies on the fact that rearrangements preserve $L_x^r$-norms and do not increase the $\dot{H}_x^1$ norm, along with the Riesz rearrangement inequality, which guarantees that
\[
\int \tfrac{a}{|x|^{2}} |f^*|^2 \,dx \leq \int \tfrac{a}{|x|^{2}} |f|^2\,dx\qtq{for}a\leq 0.
\]
Thus, we may assume that each $f_n$ is radial.  This restores the lack of compactness due to translations; one should compare this with the proof below that optimizers do \emph{not} exist for $a>0$.

Next, choose $\mu_n\in\R$ and $\lambda_n\in\R$ so that $g_n(x):=\mu_n f_n(\lambda_nx)$ satisfy $\|g_n\|_{L_x^2}=\|g_n\|_{\dot H_a^1}=1$. Note that $J_a(f_n)=J_a(g_n)$.  As $H^1_{rad}\hookrightarrow L^4_x$ compactly, passing to a subsequence we may assume that $g_n$ converges to some $g\in H_a^1$ strongly in $L_x^4$ as well as weakly in $H_a^1$.  As $g_n$ is an optimizing sequence, we deduce that $C_a=\|g\|_{L_x^4}^4$.  We also have that $\|g\|_{L_x^2} = \|g\|_{\dot H_a^1} = 1$, or else $g$ would be a super-optimizer. Thus $g$ is an optimizer.

The Euler--Lagrange equation for $g$ is given by
\[
-3C_a\L g - C_a g + 4g^3 = 0.
\]
Thus, if we define $Q_a$ via
\[
g(x) = \alpha Q_a(\lambda x), \qtq{with} \alpha = \tfrac12\sqrt{C_a} \qtq{and} \lambda = \tfrac{1}{\sqrt{3}},
\]
then $Q_a$ is an optimizer of \eqref{E:GN} that solves \eqref{elliptic}. This proves the result for $a\leq 0$.

Next, fix $a>0$ and consider a sequence $\{x_n\}\subset\R^3$ with $|x_n|\to\infty$. By \eqref{coo4} and the sharp Gagliardo--Nirenberg inequality for $a=0$, we find
\[
J_a(Q_0(x-x_n)) \to J_0(Q_0)=C_0.
\]
Thus $C_0\leq C_a$. However, for any $u\in H_x^1\backslash\{0\}$ and $a>0$, the sharp Gagliardo--Nirenberg inequality for $a=0$ implies
\[
\|u\|_{L_x^4}^4 \leq C_0\|u\|_{L_x^2} \| u\|_{\dot H_x^1}^3 < C_0\|u\|_{L_x^2} \|u\|_{\dot H_a^1}^3,
\]
whence $J_a(u)<C_0$. Thus $C_a=C_0$, and the last estimate also shows that equality is never attained. This completes the proof of Theorem~\ref{T:GN}.
\end{proof}

\begin{remark} Fix $a\leq 0$ and let $Q_a$ be as in the proof of Theorem~\ref{T:GN}.  Multiplying \eqref{elliptic} by $Q_a$ and $x\cdot\nabla Q_a$ and integrating leads to the Pohozaev identities
\begin{align*}
\| Q_a\|_{\dot H_a^1}^2+  \|Q_a\|_{L_x^2}^2 - \|Q_a\|_{L_x^4}^4 =  \| Q_a\|_{\dot H_a^1}^2 + 3\|Q_a\|_{L_x^2}^4 - \tfrac32 \|Q_a\|_{L_x^4}^4 = 0.
\end{align*}
In particular,
\begin{equation}\label{poho}
\|Q_a\|_{L_x^2}^2 = \tfrac13 \|Q_a\|_{\dot H_a^1}^2 = \tfrac14 \|Q_a\|_{L_x^4}^4
\end{equation}
and
\begin{equation}\label{var-ca}
C_a = 4\cdot 3^{-\frac32} \|Q_a\|_{L_x^2}^{-2}.
\end{equation}
\end{remark}

The next corollary shows that the thresholds defined for \eqref{nls} are always smaller than the thresholds for $a=0$.  This fact will be used in the proof of Theorem~\ref{T:exist}.

\begin{corollary}[Comparison of thresholds] \label{C:thresholds} For any $a>-\frac14$, we have
\[
\E_a \leq \E_0\qtq{and} \K_a\leq \K_0,
\]
where we recall from \eqref{E:threshold} that $\E_a := \tfrac{8}{27}C_a^{-2}$ and $\K_a := \tfrac{4}{3} C_a^{-1}.$
\end{corollary}
\begin{proof} There is nothing to check when $a\geq 0$, as $\E_a=\E_0$ and $\K_a=\K_0$ by definition.

For $a<0$, we note
\[
\| Q_0\|_{L_x^4}^4 = C_0 \|Q_0\|_{L_x^2} \| Q_0\|_{\dot H_x^1}^3 > C_0 \|Q_0\|_{L_x^2} \|Q_0\|_{\dot H_a^1}^3,\qtq{i.e.}
J_a(Q_0) >C_0,
\]
which implies $C_0 < C_a$.  The result follows.
\end{proof}

The next proposition connects the sharp Gagliardo--Nirenberg inequality with quantities appearing in the virial identities of Section~\ref{S:virial}.

\begin{proposition}[Coercivity]\label{P:coercive} Fix $a>-\frac14$. Let $u:I\times\R^3\to\C$ be the maximal-lifespan solution to \eqref{nls} with $u(t_0)=u_0\in H_a^1\backslash\{0\}$ for some $t_0\in I$. Assume that
\begin{equation}\label{quant-below}
M(u_0)E_a(u_0) \leq (1-\delta)\E_a\qtq{for some}\delta>0.
\end{equation}
Then there exist $\delta'=\delta'(\delta)>0$, $c=c(\delta,a,\|u_0\|_{L_x^2})>0$, and $\eps=\eps(\delta)>0$ such that:
\begin{itemize}
\item[a.]If $\|u_0\|_{L_x^2} \| u_0\|_{\dot H_a^1} \leq \K_a$, then for all $t\in I$,
\begin{itemize}
\item[(i)] $\|u(t)\|_{L_x^2} \|u(t)\|_{\dot H_a^1} \leq (1-\delta')\K_a$,
\item[(ii)] $\|u(t)\|_{\dot H_a^1}^2 - \tfrac34 \|u(t)\|_{L_x^4}^4 \geq c\|u(t)\|_{\dot H_a^1}^2$
\item[(iii)] $(\tfrac16+\tfrac{\delta'}{3}) \|u(t)\|_{\dot H_a^1}^2 \leq E_a(u) \leq \tfrac12 \| u(t)\|_{\dot H_a^1}^2,$
\end{itemize}
\item[b.] If $\|u_0\|_{L_x^2} \|u_0\|_{\dot H_a^1} \geq \K_a$, then for all $t\in I$,
\begin{itemize}
\item[(i)] $\|u(t)\|_{L_x^2} \|u(t)\|_{\dot H_a^1} \geq (1+\delta')\K_a$,
\item[(ii)] $(1+\eps)\|u(t)\|_{\dot H_a^1}^2 - \tfrac34 \|u(t)\|_{L_x^4}^4 \leq -c < 0.$
\end{itemize}
\end{itemize}
\end{proposition}

\begin{proof} Throughout the proof, it will be useful to recall that
\begin{equation}\label{eaka}
\E_a = \tfrac{8}{27}C_a^{-2}\qtq{and} \K_a = \tfrac43 C_a^{-1},
\end{equation}
where $C_a$ denotes the sharp constant in the Gagliardo--Nirenberg inequality \eqref{E:GN}.  By the sharp Gagliardo--Nirenberg inequality, conservation of mass and energy, and \eqref{quant-below}, we may write
\[
(1-\delta)\E_a \geq M(u) E_a(u) \geq \tfrac12 \|u(t)\|_{L_x^2}^2 \|u(t)\|_{\dot H_a^1}^2 - \tfrac14 C_a\|u(t)\|_{L_x^2}^3 \|u(t)\|_{\dot H_a^1}^3
\]
for any $t\in I$. Using \eqref{eaka}, this inequality becomes
\[
1-\delta\geq 3\biggl(\frac{\|u(t)\|_{L_x^2}\|u(t)\|_{\dot H_a^1}}{\K_a}\biggr)^2 - 2\biggl(\frac{\|u(t)\|_{L_x^2} \|u(t)\|_{\dot H_a^1}}{\K_a}\biggr)^3.
\]
Claims a.(i) and b.(i) now follow from a continuity argument, together with the observation that
\[
(1-\delta) \geq 3y^2 - 2y^3 \implies |y-1|\geq \delta' \qtq{for some} \delta'=\delta'(\delta)>0.
\]

For claim a.(iii), the upper bound follows immediately, since the nonlinearity is focusing.  For the lower bound, we again rely on the sharp Gagliardo--Nirenberg. Using a.(i)  and \eqref{eaka} as well, we find
\begin{align*}
E_a(u) &  \geq \tfrac12\|u(t)\|_{\dot H_a^1}^2 [ 1 - \tfrac12 C_a \|u(t)\|_{L_x^2} \|u(t) \|_{\dot H_a^1} ] \\
& \geq  \tfrac12\|u(t)\|_{\dot H_a^1}^2 [ 1 - \tfrac12 C_a\K_a(1-\delta')] \geq (\tfrac16+\tfrac{\delta'}{3})\|u(t)\|_{\dot H_a^1}^2
\end{align*}
for all $t\in I$. Thus a.(iii) holds.

We turn to a.(ii) and b.(ii). We begin by writing
\begin{align*}
\|u(t)\|_{\dot H_a^1}^2 - \tfrac34 \|u(t)\|_{L_x^4}^4& = 3E_a(u) - \tfrac12 \|u(t)\|_{\dot H_a^1}^2, \\
(1+\eps)\|u(t)\|_{\dot H_a^1}^2 - \tfrac34 \|u(t)\|_{L_x^4}^4&= 3E_a(u) - (\tfrac12-\eps)\|u(t)\|_{\dot H_a^1}^2,
\end{align*}
for $t\in I$, where $\eps>0$ will be chosen below. Thus a.(ii) follows from a.(iii) (choosing any $0<c\leq \delta'$).  For b.(ii), we instead rely on the conservation of mass and energy, \eqref{quant-below}, \eqref{eaka}, and
b.(i) to estimate
\begin{align*}
3E_a(u) - (\tfrac12-\eps)\|u(t)\|_{\dot H_a^1}^2 &\leq \tfrac{1}{M(u)}\bigl[3\E_a - (\tfrac12-\eps)(1+\delta')^2 \K_a^2\bigr]   \\
& \leq \tfrac{-8[(1+\delta')^2-1]+ 16\eps(1+\delta')^2}{9C_a^2 M(u)} <0,
\end{align*}
provided $\eps$ is sufficiently small depending on $\delta'$.  Thus b.(ii) follows.  \end{proof}

\begin{remark}\label{R:coercive} Suppose $u_0\in H_x^1\backslash\{0\}$ satisfies $M(u_0)E_a(u_0)<\E_a$ and $\|u_0\|_{L_x^2} \|u_0\|_{\dot H_a^1} \leq \K_a.$ Then by continuity, the maximal-lifespan solution $u$ to \eqref{nls} with initial data $u_0$ obeys $\|u(t)\|_{L_x^2} \|u(t)\|_{\dot H_a^1} < \K_a$ for all $t$ in the lifespan of $u$.  In particular, $u$ remains bounded in $H_x^1$ and hence (by Theorem~\ref{T:LWP}) is global.
\end{remark}

\section{Blowup}\label{S:blowup}

In this section, we prove Theorem~\ref{T:main}(ii); that is, we prove finite-time blowup above the threshold.  We employ the standard arguments, as in \cite{DHR, Glassey, HR, OT}. In particular, we use the virial identities of Section~\ref{S:virial}, along with Proposition~\ref{P:coercive}.

\begin{theorem}[Blowup] Let $a>-\frac14$ and $u_0\in H_x^1\backslash\{0\}$.  Suppose
\[
M(u_0)E_a(u_0) < \E_a\qtq{and} \|u_0\|_{L_x^2}\|u_0\|_{\dot H_a^1} > \K_a.
\]
Let $u:I\times\R^3\to\C$ denote the maximal-lifespan solution to \eqref{nls} with $u(0)=u_0$. If $xu_0\in L_x^2$ or $u_0$ is radial, then $u$ blows up in finite time in both time directions.
\end{theorem}

\begin{proof} Choose $\delta>0$ so that $M(u_0)E_a(u_0)\leq (1-\delta)\E_a$.

First, suppose $xu_0\in L_x^2$.  Using Lemma~\ref{L:virial0} and Proposition~\ref{P:coercive}b.(iii),
\[
\partial_{tt} \int_{\R^3} |x|^2 |u(t,x)|^2  \,dx \leq -c < 0 \qtq{for all}t\in I
\]
for some $c=c(\delta, a, \|u_0\|_{L_x^2})$.  By the standard convexity arguments (cf. \cite{Glassey}), it follows that $u$ blows up in finite time in both time directions.

Next, suppose that $u_0$ is radial. We start by noting the following radial improvement of the Gagliardo--Nirenberg inequality, which is a consequence of H\"older's inequality, radial Sobolev embedding, and the equivalence of Sobolev spaces: for any radial $f\in H_a^1$ and any $R>1$,
\[
\|f\|_{L_x^4(\{|x|>R\})}^4 \lesssim R^{-2} \|f\|_{L_x^2}^3 \| f\|_{\dot H_a^1}.
\]
Now take $R>1$ to be determined below and define $w_R\geq 0$ as in Section~\ref{S:virial}.  Using Lemma~\ref{L:virial} and the conservation of mass, we can bound
\begin{align*}
&\partial_{tt}\int_{\R^3} w_R(x) |u(t,x)|^2\,dx \leq 8\bigl[\|u(t)\|_{\dot H_a^1}^2 - \tfrac34\|u(t)\|_{L_x^4}^4\bigr] + e(t), \quad\text{where}\\
&|e(t)| \lesssim R^{-2} \|u_0\|_{L_x^2}^2 + \|u(t)\|_{L_x^4(\{|x|\geq R\})}^4.
\end{align*}
Take $\eps=\eps(\delta)>0$ and $c=c(\delta,a,\|u_0\|_{L_x^2})>0$ as in Proposition~\ref{P:coercive}b.(iii).  By the radial Gagliardo--Nirenberg inequality, conservation of mass, and Young's inequality, we may bound
\[
\|u(t)\|_{L_x^4(\{|x|\geq R\})}^4 \leq 8\eps \|u(t)\|_{\dot H_a^1}^2 + C\eps^{-1} R^{-4} \|u_0\|_{L_x^2}^6\qtq{for some}C>0.
\]
Thus, using Proposition~\ref{P:coercive}b.(iii) and choosing $R=R(c, \eps, \|u_0\|_{L_x^2})$ sufficiently large, we can guarantee that
\[
\partial_{tt} \int_{\R^3} w_R(x) |u(t,x)|^2\,dx \leq - \tfrac c2 <0,
\]
which again implies that $u$ must blow up in finite time in both time directions.\end{proof}

\section{Concentration compactness}\label{S:cc}

In this section, we prove a linear profile decomposition adapted to the $H_x^1\to L_{t,x}^5$ Strichartz inequality for $e^{-it\L}$. This result will play a key role in proving the existence of minimal blowup solutions (Theorem~\ref{T:exist}).

\begin{proposition}[Linear profile decomposition]\label{P:LPD} Fix $a>-\frac14$ and let $\{f_n\}$ be a bounded sequence in $H_a^1(\R^3)$. Passing to a subsequence, there exist $J^*\in \{0,1,2,\dots,\infty\}$, non-zero profiles $\{\phi^j\}_{j=1}^{J^*}\subset H_x^1(\R^3)$, and parameters $\{(t_n^j,x_n^j)\}_{j=1}^{J^*}\subset\R\times\R^3$ satisfying the following:

For each finite $0\leq J\leq J^*$, we can write
\begin{equation}\label{E:LPD}
f_n=\sum_{j=1}^J \phi_n^j + r_n^J,\qtq{with} \phi_n^j = [e^{it_n^j \L^{n_j}}\phi^j](x-x_n^j) \qtq{and} r_n^J\in H_x^1.
\end{equation}
Here $\L^{n_j}$ is as in Definition~\ref{D:ops}, corresponding to the sequence $\{x_n^j\}_{n=1}^\infty$.

Moreover, for each finite $0\leq J\leq J^*$ we have the following decouplings:
\begin{gather}
\label{decouple1}
\lim_{n\to\infty} \bigl\{ \|(\L)^{\frac{s}{2}} f_n\|_{L_x^2}^2- \sum_{j=1}^J \|(\L)^{\frac{s}{2}}\phi_n^j\|_{L_x^2}^2 - \|(\L)^{\frac{s}{2}}r_n^J\|_{L_x^2}^2 \bigr\}=0,\ s\in\{0,1\},\\
\label{decouple2}
\lim_{n\to\infty} \bigl\{\| f_n\|_{L_x^4}^4 - \sum_{j=1}^J \|\phi_n^j\|_{L_x^4}^4 - \| r_n^J\|_{L_x^4}^4\bigr\}=0.
\end{gather}

The remainder $r_n^J$ satisfies
\begin{equation}\label{rnJweak}
\bigl(e^{-it_n^J\L} r_n^J\bigr)(x+x_n^J) \rightharpoonup 0 \qtq{weakly in} H_x^1
\end{equation}
and vanishes in the Strichartz norm:
\begin{equation}\label{rnJ}
\lim_{J\to J^*}\limsup_{n\to\infty}\|e^{-it\L}r_n^J\|_{L_{t,x}^5(\R\times\R^3)} = 0.
\end{equation}

The parameters $(t_n^j,x_n^j)$ are asymptotically orthogonal: for any $j\neq k$,
\begin{equation}\label{orthogonal}
\lim_{n\to\infty} \bigl(|t_n^j - t_n^k| + |x_n^j - x_n^k| \bigr)= \infty.
\end{equation}

Furthermore, for each $j$, we may assume that either $t_n^j\to\pm\infty$ or $t_n^j\equiv 0$, and either $|x_n^j|\to\infty$ or $x_n^j\equiv 0$.
\end{proposition}

We prove Proposition~\ref{P:LPD} by induction, extracting one bubble of concentration at a time.  Thus the key is to isolate a single bubble, which is the content of the inverse Strichartz inequality, Proposition~\ref{P:IS}.  As a first step, we prove the following refined Strichartz estimate, which allows us to identify a scale that is responsible for concentration.

\begin{lemma}[Refined Strichartz]\label{L:RS} Let $a>-\frac14$ and $f\in \dot H_a^{\frac12}(\R^3)$. Then
\[
\|e^{-it\L} f\|_{L_{t,x}^5(\R\times\R^3)} \lesssim \|f\|_{\dot H_a^{\frac12}(\R^3)}^{\frac35}
    \sup_{N\in 2^\mathbb{Z}} \|e^{-it\L} P_N^a f\|_{L_{t,x}^5(\R\times\R^3)}^{\frac25}.
\]
\end{lemma}
\begin{proof} Denote $f_N =  P_N^a f$. We take all space-time norms over $\R\times\R^3$.  Using the square function estimate (Lemma~\ref{T:sq}), Bernstein, Strichartz, and Cauchy--Schwarz, we estimate
\begin{align*}
\|&e^{-it\L}f\|_{L_{t,x}^5}^5 \\
&\lesssim \iint \biggl[\sum_N |e^{-it\L}f_N|^2\biggr]^{\frac52}\,dx\,dt \\
&\lesssim \iint \biggl[\sum_N |e^{-it\L}f_N|^2\biggr]^{\frac12}\sum_{N_1\leq N_2} |e^{-it\L} f_{N_1}|^2 |e^{-it\L}f_{N_2}|^2 \,dx\,dt \\
&\lesssim \|e^{-it\L} f\|_{L_{t,x}^5} \sum_{N_1\leq N_2} \|e^{-it\L}f_{N_1}\|_{L_t^5 L_x^6}
    \|e^{-it\L}f_{N_2}\|_{L_t^5 L_x^{\frac{30}{7}}}  \prod_{j=1}^2 \|e^{-it\L}f_{N_j}\|_{L_{t,x}^5} \\
& \lesssim \|f\|_{\dot H_a^{\frac12}} \sup_{N}\|e^{-it\L}f_N\|_{L_{t,x}^5}^2 \sum_{N_1\leq N_2} (\tfrac{N_1}{N_2})^{\frac{1}{10}}
   \prod_{j=1}^2 \|  e^{-it\L} f_{N_j} \|_{L_t^5 \dot H_a^{\frac12,\frac{30}{11}}} \\
& \lesssim \|f\|_{\dot H_a^{\frac12}} \sup_{N}\|e^{-it\L}f_N\|_{L_{t,x}^5}^2 \sum_{N_1\leq N_2} (\tfrac{N_1}{N_2})^{\frac{1}{10}}
    \|f_{N_1}\|_{\dot H_a^{\frac12}} \|f_{N_2} \|_{\dot H_a^{\frac12}} \\
& \lesssim \|f\|_{\dot H_a^{\frac12}}^3 \sup_{N}\|e^{-it\L}f_N\|_{L_{t,x}^5}^2.
\end{align*}
This completes the proof of Lemma~\ref{L:RS}.  \end{proof}

We turn to the inverse Strichartz inequality.

\begin{proposition}[Inverse Strichartz]\label{P:IS}
Fix $a>-\tfrac14$ and let $\{f_n\}\subset H^1_a(\R^3)$ satisfy
\[
\lim_{n\to\infty} \|f_n\|_{H^1_a} = A<\infty \qtq{and} \lim_{n\to\infty} \| e^{-it\L} f_n\|_{L_{t,x}^5} = \eps > 0.
\]
Passing to a subsequence, there exist $\phi\in H_x^1$ and $\{(t_n,x_n)\}\subset\R\times\R^3$ such that
\begin{align}
\label{weak}
& g_n(\cdot) = [e^{-it_n\L} f_n](\cdot+x_n) \rightharpoonup \phi(\cdot) \qtq{weakly in} H_x^1,  \\
\label{LB} & \|\phi\|_{H_a^1} \gtrsim \eps(\tfrac{\eps}{A})^{14}.
\end{align}

Furthermore, defining
\[
\phi_n(x) = e^{it_n \L}[\phi(\cdot-x_n)](x)=[e^{it_n\L^n}\phi](x-x_n),
\]
with $\L^n$ as in Definition~\ref{D:ops}, we have
\begin{align}
& \lim_{n\to\infty} \bigl\{ \|(\L)^{\frac{s}{2}} f_n\|_{L_x^2}^2 - \|(\L)^{\frac{s}{2}} (f_n-\phi_n)\|_{L_x^2}^2 - \|(\L)^{\frac{s}{2}} \phi_n\|_{L_x^2}^2 \bigr\} =0,\ s\in\{0,1\}, \label{H1-decouple} \\
& \lim_{n\to\infty} \bigl\{ \|f_n\|_{L_x^4}^4 - \|f_n-\phi_n\|_{L_x^4}^4  - \|\phi_n\|_{L_x^4}^4\bigr\} = 0.    \label{L4-decouple}
\end{align}

Finally, we may assume that either $t_n\to\pm\infty$ or $t_n\equiv 0$, and either $|x_n|\to\infty$ or $x_n\equiv 0$.
\end{proposition}

\begin{proof} By Lemma~\ref{L:RS}, for $n$ sufficiently large, there exists $N_n\in 2^{\mathbb{Z}}$ such that
\[
\| e^{-it\L}  P_{N_n}^a f_n \|_{L_{t,x}^5} \gtrsim \eps(\tfrac{\eps}{A})^{\frac32}.
\]
Note that by Bernstein and Strichartz, we may bound
\[
\| e^{-it\L}  P_N^a f_n \|_{L_{t,x}^5} \lesssim (N^{\frac12}\vee N^{-\frac12}) A\qtq{for any} N\in 2^{\mathbb{Z}},
\]
so that we must have
\[
(\tfrac{\eps}{A})^5 \lesssim N_n \lesssim (\tfrac{A}{\eps})^5.
\]
Passing to a subsequence, we may assume $N_n\equiv N_*$. Thus
\[
\| e^{-it\L} P_{N_*}^a f_n \|_{L_{t,x}^5} \gtrsim \eps (\tfrac{\eps}{A})^{\frac32}
\]
for all $n$ sufficiently large.  In what follows, we use the shorthand:
\[
 P_* :=  P_{N_*}^a.
\]

Note that by H\"older and Bernstein, for any $N>0$ and $c>0$ we have
\[
\| P_N^a F\|_{L_x^5(\{|x|\leq c N^{-1}\})} \lesssim \| P_N^a F\|_{L_x^6} \|1\|_{L_x^{30}(\{|x|\leq c N^{-1}\})}
    \lesssim c^{\frac{1}{10}} \|P_N^a F\|_{L_x^5}.
\]
Thus,
\[
\|e^{-it\L} P_* f_n \|_{L_{t,x}^5(\R\times\{|x|\geq \alpha \})}\gtrsim \eps(\tfrac{\eps}{A})^{\frac32},\qtq{provided} \alpha = c N_*^{-1}
\]
for $c>0$ sufficiently small. Using this together with H\"older, Strichartz, and Bernstein, we find
\begin{align*}
\eps(\tfrac{\eps}{A})^{\frac32} & \lesssim \| e^{-it\L} P_* f_n \|_{L_{t,x}^\infty(\R\times \{|x|\geq\alpha\})}^{\frac13} \|e^{-it\L} P_* f_n \|_{L_{t,x}^{\frac{10}{3}}}^{\frac23}  \\
& \lesssim \| e^{-it\L} P_* f_n \|_{L_{t,x}^\infty(\R\times \{|x|\geq\alpha \})}^{\frac13}
    \| f_n\|_{L_x^2}^{\frac23},
\end{align*}
and hence there exist $(\tau_n,x_n)\in\R\times\R^3$ with $|x_n|\geq \alpha$ such that
\begin{equation}\label{bubble}
 \biggl| [e^{-i\tau_n\L}  P_* f_n](x_n)\biggr| \gtrsim \eps (\tfrac{\eps}{A})^{\frac{13}2}.
\end{equation}
Passing to a subsequence, we may assume $\tau_n\to\tau_\infty\in [-\infty,\infty]$. If $\tau_\infty\in\R$, we set $t_n\equiv 0$. If $\tau_\infty\in\{\pm\infty\}$, we set $t_n=\tau_n$. We may also assume that $x_n\to x_\infty\in\R^3\backslash\{0\}$ or $|x_n|\to\infty$.

We now let
\[
g_n(x) = [e^{-it_n\L} f_n](x+x_n),\qtq{i.e.} f_n(x) = [e^{it_n\L^n}g_n](x-x_n),
\]
where $\L^n$ is as in Definition~\ref{D:ops}.  Note that
\[
\|g_n\|_{H_x^1} = \|e^{-it_n\L}f_n\|_{H_x^1}\sim \|e^{-it_n\L}f_n\|_{H_a^1} \lesssim A,
\]
so that $g_n$ converges weakly to some $\phi$ in $H_x^1$ (up to a subsequence). Define
\[
\phi_n(x) = e^{it_n\L}[\phi(\cdot-x_n)](x) = [e^{it_n\L^n}\phi](x-x_n).
\]
By a change of variables and weak convergence, we have
\begin{equation*}
\|f_n\|_{L_x^2}^2 - \|f_n - \phi_n \|_{L_x^2}^2 - \|\phi_n\|_{L_x^2}^2 = 2\Re\langle g_n, \phi\rangle - 2\langle\phi,\phi\rangle \to 0
\end{equation*}
as $n\to\infty$; using \eqref{coo1} as well, we get
\begin{align*}
\|f_n\|_{\dot H^1_a}^2 - \|f_n-\phi_n\|_{\dot H^1_a}^2 - \|\phi_n\|_{\dot H^1_a}^2 & = 2\Re\langle g_n,\L^n \phi\rangle - 2\langle\phi, \L^n \phi\rangle \\
& \to 2[\langle\phi,\L^\infty\phi\rangle - \langle\phi,\L^\infty\phi\rangle] = 0.
\end{align*}
This proves \eqref{H1-decouple}.

We next turn to \eqref{LB}. We define
\[
h_n = \begin{cases}
     P_*^n \delta_0 & \text{if }\tau_\infty \in \{\pm \infty\}, \\
    e^{-i\tau_n \L^n} P_*^n\delta_0 & \text{if }\tau_\infty\in\R,
    \end{cases}
\qtq{where}  P_*^n = e^{-\L^n/N_*^2} - e^{-4\L^n/N_*^2}.
\]
Note that after a change of variables, \eqref{bubble} reads
\[
\bigl|\langle h_n,g_n\rangle\bigr| \gtrsim \eps (\tfrac{\eps}{A})^{\frac{13}2}.
\]
As $|x_n|\geq \alpha>0$, we have by \eqref{coo2} and \eqref{coo3} that
\[
h_n \to \begin{cases}  P_*^\infty\delta_0 &\text{if } \tau_\infty\in\{\pm\infty\}, \\
    e^{-i\tau_\infty \L^\infty} P_*^\infty \delta_0 &\text{if } \tau_\infty\in\R,  \end{cases}
\qtq{where}  P_*^\infty = e^{-\L^\infty/N_*^2} - e^{-4\L^\infty/N_*^2}.
\]
Here the convergence holds strongly in $\dot H_x^{-1}$.  Thus, if $\tau_\infty\in\{\pm\infty\}$, we have
\[
\eps (\tfrac{\eps}{A})^{\frac{13}2} \lesssim |\langle  P_*^\infty \delta_0, \phi\rangle|\lesssim \|\phi\|_{L_x^2} \|  P_*^\infty \delta_0\|_{L_x^2}.
\]

By the heat kernel bounds of Lemma~\ref{L:kernel}, we can bound
\[
\|  P_*^\infty \delta_0\|_{L_x^2}  \lesssim N_*^{\frac32}\lesssim  (\tfrac{A}{\eps})^{\frac{15}{2}},
\]
which implies
\[
\| \phi \|_{L_x^2} \gtrsim \eps(\tfrac{\eps}{A})^{14}.
\]
The case of $\tau_\infty\in\R$ is similar.  This proves \eqref{LB}.

We now turn to \eqref{L4-decouple}. Using Rellich--Kondrashov and passing to a subsequence, we may assume $g_n\to\phi$ almost everywhere.  Thus, Lemma~\ref{L:RF} implies
\[
\|g_n\|_{L_x^4}^4 - \|g_n-\phi\|_{L_x^4}^4 - \|\phi\|_{L_x^4}^4 \to 0.
\]
This, together with a change of variables, gives \eqref{L4-decouple} in the case $t_n\equiv 0$.  If instead $t_n\to\pm\infty$, then \eqref{L4-decouple} follows from Corollary~\ref{L4-to-zero}.

Finally, if $x_n\to x_\infty\in\R^3$, then we may take $x_n\equiv 0$ by replacing $\phi(\cdot)$ with $\phi(\cdot-x_\infty)$.  To see that all of the conclusions still hold, one relies on the continuity of translation in the strong $H_x^1$-topology.
\end{proof}

With Proposition~\ref{P:IS} in hand, the proof of Proposition~\ref{P:LPD} is relatively straightforward.  We give only a quick sketch; for details, one can refer to \cite{Oberwolfach}.

\begin{proof} The statements \eqref{E:LPD}, \eqref{decouple1}, and \eqref{decouple2} follow by induction.  Specifically, one sets $r_n^0 = f_n$ and then applies Proposition~\ref{P:IS} to the sequence $r_n^0$ to find $\phi_n^1$. Proceeding inductively, one applies Proposition~\ref{P:IS} to the sequences $r_n^J = r_n^{J-1} - \phi_n^J$ to deduce \eqref{E:LPD}. The decouplings \eqref{H1-decouple} and \eqref{L4-decouple} imply \eqref{decouple1} and \eqref{decouple2}.  We terminate the process at a finite $J^*$ if $\lim_{n\to\infty} \|e^{-it\L} r_n^{J^*}\|_{L_{t,x}^5} = 0$.

We define
\[
\eps_J = \lim_{n\to\infty}\|e^{-it\L}r_n^J\|_{L_{t,x}^5}\qtq{and} A_J = \lim_{n\to\infty}\|r_n^J\|_{H_a^1}.
\]
That $\eps_J\to 0$ follows from \eqref{decouple1}, \eqref{LB}, and the fact that $A_J\leq A_0$; indeed,
\[
\sum_{j=1}^J \eps_{j-1}^2(\tfrac{\eps_{j-1}}{A_0})^{\frac28}
    \lesssim \sum_{j=1}^J \|\phi_n^j\|_{H_a^1}^2 \lesssim A_0^2,
\]
which gives \eqref{rnJ}.

We next turn to \eqref{rnJweak}.  Note that by \eqref{weak}, we have
\begin{equation}\label{weakJ}
[e^{-it_n^J \L}r_n^{J-1}](\cdot+x_n^J) \rightharpoonup \phi^J\qtq{for each finite}J\geq 1.
\end{equation}
Recalling that $r_n^J = r_n^{J-1} - \phi_n^J$, we deduce \eqref{rnJweak}.

Finally, for \eqref{orthogonal}, we give $(j,k)$ the lexicographical order and suppose towards a contradiction that \eqref{orthogonal} fails for the first time at some $(j,k)$. We may assume $j<k$, and hence
\begin{align}\label{parameters converge}
&t_n^j - t_n^k\to t_0\in\R,\quad x_n^j - x_n^k\to x_0\in\R^3,\quad\text{and}
\\ \label{pf-orthogonal2}
&|x_n^j - x_n^\ell |\to\infty\qtq{or} |t_n^j - t_n^\ell|\to\infty\qtq{for each} j<\ell<k.
\end{align}
Using the construction and \eqref{weakJ}, we deduce
\begin{equation}\label{pf-orthogonal}
[e^{-it_n^k\L}r_n^j](\cdot + x_n^k) - \sum_{\ell=j+1}^{k-1} [e^{-it_n^k\L}\phi_n^\ell](\cdot+x_n^k) \rightharpoonup \phi^k(\cdot).
\end{equation}
We claim that both terms above converge weakly to zero, contradicting that $\phi^k\neq 0$.

Indeed, for the first term in \eqref{pf-orthogonal}, we rely on \eqref{rnJweak} and \eqref{parameters converge}, together with the following lemma from \cite{KMVZZ2}:
\begin{lemma}[Weak convergence, \cite{KMVZZ2}] Let $f_n\in\dot H_x^1$ satisfy $f_n\rightharpoonup 0$ weakly in $\dot H_x^1$, and suppose $\tau_n\to\tau_0\in\R$. Then for any sequence $\{y_n\}\subset\R^3$, we have
\[
e^{-i\tau_n\L^n} f_n\rightharpoonup 0 \qtq{weakly in} \dot H_x^1.
\]
Here $\L^n$ is as in Definition~\ref{D:ops}, corresponding to the sequence $\{y_n\}$.
\end{lemma}
For the second term in \eqref{pf-orthogonal}, we rely on \eqref{pf-orthogonal2}, along with another lemma imported from \cite{KMVZZ2}:
\begin{lemma}[Weak convergence, \cite{KMVZZ2}] Let $f\in \dot H_x^1$. Let $\{(\tau_n,x_n)\}\subset\R\times\R^3$ and suppose that either $|\tau_n|\to\infty$ or $|x_n|\to\infty$. Then for any sequence $\{y_n\}\subset\R^3$, we have
\[
[e^{-i\tau_n\L^n}f](\cdot+x_n)\rightharpoonup 0 \qtq{weakly in} \dot H_x^1,
\]
where $\L^n$ is as in Definition~\ref{D:ops}, corresponding to the sequence $\{y_n\}$.
\end{lemma}
This completes the proof of Proposition~\ref{P:LPD}. \end{proof}

\section{Embedding nonlinear profiles}\label{S:embed}

In this section, we show that for profiles appearing in Proposition~\ref{P:LPD} living far from the origin, we can construct scattering solutions to \eqref{nls} with these profiles as data, provided the profiles are below the threshold for the standard NLS given in \cite{DHR, HR}.  This result will play an important role in the construction of minimal blowup solutions (Theorem~\ref{T:exist}).

Recall the definition of the thresholds $\E_a$ and $\K_a$ in \eqref{E:threshold}.  In particular, $\E_0$ and $\K_0$ refer to the threshold for the standard cubic NLS appearing in \cite{DHR,HR}; the results of \cite{DHR,HR} are encapsulated in Theorem~\ref{T:DHR}.

\begin{theorem}[Embedding nonlinear profiles]\label{T:embedding} Fix $a>-\frac14$ and let $\{t_n\}\subset\R$ satisfy $t_n\equiv 0$ or $t_n\to\pm\infty$, and let $\{x_n\}\subset\R^3$ satisfy $|x_n|\to\infty$. Let $\phi\in H_x^1(\R^3)$ satisfy
\begin{equation} \label{embedding-threshold}
\begin{aligned}
M(\phi)E_{0}(\phi) < \E_0 \qtq{and}\|\phi\|_{L_x^2}\|\phi\|_{\dot H_x^1} < \K_0 
&\qtq{if} t_n\equiv 0, \\
\tfrac12\|\phi\|_{L_x^2}^2 \|\phi\|_{\dot H_x^1}^2 < \E_0 &\qtq{if} t_n\to\pm\infty.
\end{aligned}
\end{equation}
Define
\[
\phi_n(x) = [e^{-it_n\L^n}\phi](x-x_n),
\]
where $\L^n$ is as in Definition~\ref{D:ops}. Then for all $n$ sufficiently large, there exists a global solution $v_n$ to \eqref{nls} with $v_n(0)=\phi_n$ satisfying
\[
\| v_n\|_{S_a^1(\R)} \lesssim 1,
\]
with the implicit constant depending on $\|\phi\|_{H_x^1}$.

Furthermore, for any $\eps>0$, there exists $N_\eps\in\mathbb{N}$ and $\psi_\eps\in C_c^\infty(\R\times\R^3)$ such that for $n\geq N_\eps$,
\begin{equation}\label{embed-cc}
\|v_n-\psi_\eps(\cdot+t_n,\cdot - x_n)\|_{X(\R\times\R^3)}  < \eta,
\end{equation}
where
\begin{align*}
& X \in \{L_{t,x}^5, L_{t,x}^{\frac{10}{3}}, L_t^{\frac{30}{7}} L_x^{\frac{90}{31}}, L_t^{\frac{30}{7}} \dot H_a^{\frac{31}{60},\frac{90}{31}}\},
\end{align*}
\end{theorem}

\begin{proof} Before launching into the proof, we note that
\begin{equation}\label{phin-bd}
\|\phi_n \|_{\dot H_x^1} \lesssim \|\phi\|_{\dot H_x^1}\qtq{uniformly in }n.
\end{equation}

We begin by finding solutions to \eqref{nls0} related to $\phi$. Define $P_n = P_{\leq |x_n|^{\theta}}$ for some $0<\theta<1$; here, the notation refers to the standard Littlewood--Paley projections.  Note that since $|x_n|\to\infty$, we have that $P_n\phi$ satisfies \eqref{embedding-threshold} for all $n$ sufficiently large. Thus, we are in a position to apply Theorem~\ref{T:DHR}: If $t_n\equiv 0$, then we let $w_n$ and $w_\infty$ be the solutions to \eqref{nls0} with $w_n(0)=P_n\phi$ and $w_\infty(0) = \phi$; if $t_n\to\pm\infty$, we instead let $w_n$ and $w_\infty$ be the solutions to \eqref{nls0} satisfying
\[
\|w_n(t) - e^{it\Delta}P_n\phi\|_{H_x^1} \to 0 \qtq{and} \|w_\infty(t)-e^{it\Delta}\phi\|_{H_x^1}\to 0
\]
as $t\to\pm\infty$.  Note that in both cases, we have
\begin{equation}\label{w-good}
\|w_n\|_{S_0^1(\R)} + \|w_\infty\|_{S_0^1(\R)} \lesssim 1
\end{equation}
for $n$ sufficiently large, with the implicit constant depending on $\|\phi\|_{H_x^1}$. Also, since $\|P_n\phi - \phi\|_{H_x^1}\to 0$ as $n\to\infty$, the stability theory for \eqref{nls0} implies that
\begin{equation}\label{embed-conv}
\lim_{n\to\infty} \| w_n - w_\infty \|_{S_0^1(\R)} = 0.
\end{equation}
By persistence of regularity for \eqref{nls0} and the fact that $\| |\nabla|^s P_n \phi \|_{H_x^1} \lesssim |x_n|^{s\theta}$ for any $s\geq 0$, we have
\begin{equation}\label{embed-persist}
\| |\nabla|^s w_n \|_{S_0^1(\R)} \lesssim  |x_n|^{\theta s}\quad\text{for all }s\geq 0\text{ and }n\text{ large}.
\end{equation}
Finally, note that in either case, $w_\infty$ scatters to some asymptotic states $w_\pm$ in $H_x^1$.

We next construct approximate solutions to \eqref{nls}. To begin, for each $n$ we let $\chi_n$ be a smooth function satisfying
\begin{equation}\label{chin}
\chi_n(x) = \begin{cases} 0 & |x+x_n| \leq \tfrac14 |x_n|, \\ 1 & |x+x_n| > \tfrac12 |x_n|, \end{cases}
\qtq{with} \sup_x |\partial^\alpha \chi_n (x) | \lesssim |x_n|^{-|\alpha|}
\end{equation}
for all multi-indices $\alpha$. Note that $\chi_n(x)\to 1$ as $n\to\infty$ for each $x\in\R^3$.  Given $T>0$, we now define
\[
\tilde v_{n,T}(t,x) := \begin{cases} [\chi_n w_n](t,x-x_n) & |t| \leq T, \\
[e^{-i(t-T)\L}\tilde v_{n,T}(T)](x) & t > T, \\
[e^{-i(t+T)\L}\tilde v_{n,T}(-T)](x) & t<-T.
\end{cases}
\]

We wish to construct $v_n$ by applying Theorem~\ref{T:stab}.  To do so, we must verify the following: For $s\in\{\frac12,\frac35\}$,
\begin{align}
\label{embed-bounds}
&\limsup_{T\to\infty}\limsup_{n\to\infty}\bigl\{\|\tilde v_{n,T}\|_{L_t^\infty H_x^1}+ \|\tilde v_{n,T} \|_{L_{t,x}^5}\bigr\} \lesssim 1, \\
\label{embed-data}
&\lim_{T\to\infty} \limsup_{n\to\infty} \| \tilde v_{n,T}(t_n) - \phi_n \|_{\dot H_x^s} = 0, \\
\label{embed-approx}
&\lim_{T\to\infty} \limsup_{n\to\infty} \||\nabla|^s[(i\partial_t - \L)\tilde v_{n,T} + |\tilde v_{n,T}|^2 \tilde v_{n,T}] \|_{N(\R)} = 0,
\end{align}
where all space-time norms are over $\R\times\R^3$.  For the definition of $N(\R)$, see the statement of Theorem~\ref{T:stab}.

To begin, using Strichartz, equivalence of Sobolev spaces, and \eqref{w-good}, we have
\begin{align}
\|\tilde v_{n,T} \|_{L_t^\infty H_x^1} & \lesssim \| \langle\nabla\rangle (\chi_n w_n) \|_{L_t^\infty L_x^2}\nonumber \\
&\lesssim  \|\nabla\chi_n\|_{L_x^3}\|w_n\|_{L_t^\infty L_x^6} + \|\chi_n\|_{L_x^\infty} \|\langle\nabla\rangle w_n\|_{L_t^\infty L_x^2}\nonumber  \\
&\lesssim \|w_n\|_{L_t^\infty H_x^1}\lesssim 1  \quad\text{uniformly in } n,T,\label{embed-bound0}
\end{align}
while by Strichartz and \eqref{w-good}, we have
\[
\|\tilde v_{n,T} \|_{L_{t,x}^{5}} \lesssim 1\qtq{uniformly in}n,T.
\]
This yields \eqref{embed-bounds}.

For later use, note also that for $s\in\{\frac12,\frac35\}$,
\begin{equation}\label{embed-bd-more}
\|\tilde v_{n,T}\|_{L_t^5 \dot H_a^{s,\frac{30}{11}}}+ \|\tilde v_{n,T}\|_{L_t^{\frac{30}{7}} \dot H_a^{\frac{8}{15},\frac{90}{31}}}\lesssim 1\qtq{uniformly in}n,T,
\end{equation}
which follows from the equivalence of Sobolev spaces and \eqref{w-good}.

We turn to \eqref{embed-data}. By \eqref{phin-bd} and \eqref{embed-bound0}, we first note that
\begin{equation}\label{embed-bound1}
\| \tilde v_{n,T}(t_n) - \phi_n \|_{\dot H_x^1} \lesssim 1\qtq{uniformly in }n,T.
\end{equation}

Consider the case $t_n\equiv 0$. Then
\[
\|\tilde v_{n,T}(0)-\phi_n\|_{L_x^2} = \|\chi_n P_n\phi - \phi\|_{L_x^2},
\]
which converges to zero as $n\to\infty$ by the dominated convergence theorem and Bernstein. Thus, by interpolation with \eqref{embed-bound1}, we see that \eqref{embed-data} holds when $t_n\equiv 0$.

Now consider the case $t_n\to\infty$; the case $t_n\to-\infty$ is similar.  For sufficiently large $n$, we have $t_n>T$, and hence (since $\L^\infty = -\Delta$)
\begin{align}
\nonumber
\|\tilde v_{n,T}(t_n) - \phi_n \|_{L_x^2} & = \|e^{iT\L^n}\chi_nw_n(T) - \phi\|_{L_x^2} \\
& \lesssim \|\chi_nw_n(T) - w_\infty(T) \|_{L_x^2} \label{e-data1} \\
& \quad + \|[e^{iT\L^n}-e^{iT\L^\infty}]w_\infty(T) \|_{L_x^2} \label{e-data2} \\
&\quad + \|e^{-iT\Delta}w_\infty(T) - \phi\|_{L_x^2}. \label{e-data3}
\end{align}
Using dominated convergence and \eqref{embed-conv}, we deduce that $\eqref{e-data1}\to 0$ as $n\to\infty$.  By \eqref{coo5}, we also find that $\eqref{e-data2}\to 0$ as $n\to\infty$.  Finally, by construction, we have that $\eqref{e-data3}\to 0$ as $T\to\infty$. Interpolating with \eqref{embed-bound1}, we see that \eqref{embed-data} holds in the case $t_n\to\pm\infty$, as well. This completes the proof of \eqref{embed-data}.

We now turn to \eqref{embed-approx}. First note that for $|t|>T$, we have that
\[
e_{n,T} := (i\partial_t - \L)\tilde v_{n,T} + |\tilde v_{n,T}|^2 \tilde v_{n,T} = |\tilde v_{n,T}|^2 \tilde v_{n,T}.
\]
For $s\in\{\frac12,\frac35\}$, we estimate
\begin{align*}
\||\nabla|^s\bigl(|\tilde v_{n,T}|^2 \tilde v_{n,T}\bigr)\|_{L_t^{\frac53} L_x^{\frac{30}{23}}(\{t>T\}\times\R^3)} & \lesssim
    \|\tilde v_{n,T}\|_{L_t^5 \dot H_a^{s,\frac{30}{11}}} \|\tilde v_{n,T}\|_{L_{t,x}^5(\{t>T\}\times\R^3)}^2 \\
&\lesssim \| e^{-it\L^n}[\chi_n w_n(T)]\|_{L_{t,x}^5((0,\infty)\times\R^3)}^2.
\end{align*}
We now claim
\begin{equation}\label{embed-error1}
\lim_{T\to\infty}\limsup_{n\to\infty} \| e^{-it\L^n}[\chi_n w_n(T)]\|_{L_{t,x}^5((0,\infty)\times\R^3)} = 0,
\end{equation}
which implies \eqref{embed-approx} for times $t>T$.  (The case $t<-T$ is similar.) We use Sobolev embedding and Strichartz to estimate
\begin{align}
 \| e^{-it\L^n}[\chi_n w_n(T)]\|_{L_{t,x}^5((0,\infty)\times\R^3)} & \lesssim \|\chi_n w_n(T) - w_\infty(T)\|_{\dot H_x^{\frac12}} \label{embed-error2} \\
&\quad + \|[e^{-it\L^n}-e^{-it\L^\infty}]w_\infty(T)\|_{L_{t,x}^5((0,\infty)\times\R^3)} \label{embed-error3} \\
&\quad + \|e^{it\Delta}[w_\infty(T) - e^{iT\Delta}w_+]\|_{L_{t,x}^5((0,\infty)\times\R^3)} \label{embed-error4} \\
&\quad + \|e^{it\Delta}w_+ \|_{L_{t,x}^5((T,\infty)\times\R^3)}. \label{embed-error5}
\end{align}
Note that $\eqref{embed-error2}\to 0$ as $n\to\infty$; indeed, this follows from $\dot H_x^1$-boundedness and our analysis of \eqref{e-data1}. Next, $\eqref{embed-error3}\to 0$ as $n\to\infty$ by Corollary~\ref{S-to-zero}. We have that $\eqref{embed-error4}\to 0$ as $T\to\infty$ by Strichartz and the definition of $w_+$. Finally, $\eqref{embed-error5}\to 0$ as $T\to\infty$ by monotone convergence. This completes the proof of \eqref{embed-approx} for times $|t|> T$.

We next consider times $|t|\leq T$. For these times, we have
\begin{align}
e_n(t,x) & = [(\chi_n - \chi_n^3) |w_n|^2 w_n](t,x-x_n) \label{embed-error6} \\
&\quad  + [w_n\Delta \chi_n + 2\nabla \chi_n\cdot\nabla w_n](t,x-x_n) \label{embed-error7} \\
&\quad - \tfrac{a}{|x|^2}[\chi_n w_n](t,x-x_n). \label{embed-error8}
\end{align}

First, by Sobolev embedding and \eqref{w-good},
\begin{align*}
\|\nabla \eqref{embed-error6}\|_{L_t^{\frac53}L_x^{\frac{30}{23}}} & \lesssim \|\nabla \chi_n \|_{L_x^3} \|w_n\|_{L_{t,x}^5}^2
    \|w_n\|_{L_t^5 L_x^{30}} + \|w_n\|_{L_{t,x}^5}^2 \|\nabla w_n\|_{L_t^5 L_x^{\frac{30}{11}}} \\
&\lesssim \||\nabla|^{\frac12}w_n\|_{L_t^5 L_x^{\frac{30}{11}}}^2
\|\nabla w_n\|_{L_t^5 L_x^{\frac{30}{11}}} \lesssim 1,
\end{align*}
while
\begin{align*}
\|\eqref{embed-error6}\|_{L_t^{\frac53}L_x^{\frac{30}{23}}} \lesssim
\||\nabla|^{\frac12}w_n\|_{L_t^5 L_x^{\frac{30}{11}}}^2
\|(\chi_n^3-\chi_n)w_n\|_{L_t^5 L_x^{\frac{30}{11}}} \to 0
\end{align*}
as $n\to\infty$ by dominated convergence and \eqref{embed-conv}. Thus, by interpolation, we have
\[
\lim_{T\to\infty} \lim_{n\to\infty} \| |\nabla|^s \eqref{embed-error6} \|_{L_t^{\frac53} L_x^{\frac{30}{23}}} = 0\qtq{for}s\in\{\tfrac12,\tfrac35\}.
\]

Next, using \eqref{embed-persist},
\begin{align*}
\|\nabla \eqref{embed-error7}\|_{L_t^1 L_x^2(\{|t|\leq T\}\times\R^3)} & \lesssim T\bigl\{ \|\nabla\Delta \chi_n\|_{L_x^\infty}
    \|w_n\|_{L_t^\infty L_x^2} + \|\Delta\chi_n\|_{L_x^\infty}\|\nabla w_n\|_{L_t^\infty L_x^2}\\
&\quad\quad\quad + \|\nabla\chi_n\|_{L_x^\infty}
    \|\Delta w_n\|_{L_t^\infty L_x^2} \bigr\} \\
&\lesssim T\bigl\{ |x_n|^{-3}+|x_n|^{-2}+ |x_n|^{-1+\theta} \bigr\} \to 0 \qtq{as}n\to\infty.
\end{align*}
Similarly,
\begin{align*}
\|\eqref{embed-error7}\|_{L_t^1 L_x^2(\{|t|\leq T\}\times\R^3)} & \lesssim T\bigl\{ \|\Delta\chi_n\|_{L_x^\infty}
    \|w_n\|_{L_t^\infty L_x^2} + \|\nabla \chi_n \|_{L_x^\infty} \|\nabla w_n\|_{L_t^\infty L_x^2} \bigr\} \\
&\lesssim T\bigl\{ |x_n|^{-2} + |x_n|^{-1} \}\to 0 \qtq{as}n\to\infty,
\end{align*}
and thus, by interpolation,
\[
\lim_{T\to\infty}\lim_{n\to\infty} \||\nabla|^s \eqref{embed-error7}\|_{L_t^1 L_x^2(\{|t|\leq T\}\times\R^3)} = 0\qtq{for}s\in\{\tfrac12,\tfrac35\}.
\]
Finally, we estimate
\begin{align*}
\|\langle\nabla\rangle \eqref{embed-error8} \|_{L_t^1 L_x^2(\{|t|\leq T\}\times\R^3)}
    & \lesssim T\bigl\{ \|\tfrac{\chi_n}{|\cdot + x_n|^2}\|_{L_x^\infty}\|\langle\nabla\rangle w_n\|_{L_t^\infty L_x^2} \\
& \quad \quad + \|\nabla\bigl(\tfrac{\chi_n}{|\cdot + x_n|^2}\bigr)\|_{L_x^\infty}\|w_n\|_{L_t^\infty L_x^2} \bigr\} \\
&\lesssim T\bigl\{ |x_n|^{-2} + |x_n|^{-3} \}\to 0\qtq{as}n\to\infty,
\end{align*}
so that
\[
\lim_{T\to\infty}\lim_{n\to\infty} \| |\nabla|^s\eqref{embed-error8}\|_{L_t^1 L_x^2(\{|t|\leq T\}\times\R^3)} = 0\qtq{for}s\in\{\tfrac12,\tfrac35\}.
\]
This completes the proof of \eqref{embed-approx} for times $|t|\leq T$.

With \eqref{embed-bounds}, \eqref{embed-data}, and
\eqref{embed-approx} in place, we apply Theorem~\ref{T:stab} to
deduce the existence of a global solution $v_n$ to \eqref{nls} with
$v_n(0) = \phi_n$ satisfying
\begin{gather}
\nonumber
 \| v_n\|_{S_a^1(\R)} \lesssim 1 \qtq{uniformly in} n, \\
\label{embed-conclude-close}
 \lim_{T\to\infty} \limsup_{n\to\infty} \|[ v_n(\cdot-t_n) - \tilde v_{n,T}(\cdot)] \|_{\dot S_a^s(\R)} = 0\qtq{for}s\in\{\tfrac12,\tfrac35\}.
\end{gather}

Finally, we turn to \eqref{embed-cc}.  We will only prove the approximation in the space $L_t^{\frac{30}{7}} \dot H_x^{\frac{31}{60}, \frac{90}{31}}$.  Approximation in the other spaces is similar (in fact, simpler).

Fix $\eps>0$. As $C_c^\infty(\R\times\R^3)$ is dense in $L_t^{\frac{30}{7}} \dot H_x^{\frac{31}{60}, \frac{90}{31}}$, we may find $\psi_\eps\in C_c^\infty$ such that
\[
\|w_\infty - \psi_\eps\|_{L_t^{\frac{30}{7}} \dot H_x^{\frac{31}{60}, \frac{90}{31}}} < \tfrac{\eps}{3}.
\]
In light of \eqref{embed-conv} and \eqref{embed-conclude-close}, it suffices to show that
\begin{equation}\label{embed-cc-nts}
\| \tilde v_{n,T}(t,x) - w_\infty(t,x-x_n) \|_{L_t^{\frac{30}{7}} \dot H_x^{\frac{31}{60}, \frac{90}{31}}} < \tfrac{\eps}{3}
\end{equation}
for $n,T$ large.  Again, we will use interpolation, beginning with the following consequence of the triangle inequality and \eqref{embed-bd-more}:
\[
\| |\nabla|^{\frac{8}{15}}[\tilde v_{n,T}(t,x) - w_\infty(t,x-x_n)]\|_{L_t^{\frac{30}{7}} L_x^{\frac{90}{31}}} \lesssim 1.
\]

On the other hand, we can estimate
\begin{align*}
\| \tilde v_{n,T}(t,x) - w_\infty&(t,x-x_n)\|_{L_t^{\frac{30}{7}} L_x^{\frac{90}{31}}} \\
&\lesssim \|\chi_n w_n-w_\infty\|_{L_t^{\frac{30}{7}} L_x^{\frac{90}{31}}([-T,T]\times\R^3)} \\
&\quad + \|e^{-i(t-T)\L^n}[\chi_n w_n(T)] - w_\infty\|_{L_t^{\frac{30}{7}} L_x^{\frac{90}{31}}((T,\infty)\times\R^3)} \\
& \quad + \|e^{-i(t+T)\L^n}[\chi_n w_n(-T)] - w_\infty \|_{L_t^{\frac{30}{7}} L_x^{\frac{90}{31}}(-\infty,-T)\times\R^3)}.
\end{align*}
The first term converges to zero as $n\to\infty$ by the dominated convergence theorem and \eqref{w-good}. The second and third terms are similar, so we only consider the second. For this term, we apply the triangle inequality. By \eqref{w-good} and monotone convergence,
\[
\|w_\infty\|_{L_t^{\frac{30}{7}} L_x^{\frac{90}{31}}((T,\infty)\times\R^3)}\to 0\qtq{as}T\to\infty,
\]
while arguing as we did for \eqref{embed-error1} we see that
\[
\lim_{T\to\infty}\limsup_{n\to\infty} \| e^{-it\L^n}[\chi_n w_n(T)]\|_{L_t^{\frac{30}{7}} L_x^{\frac{90}{31}}((0,\infty)\times\R^3)}=0.
\]
Interpolation now yields \eqref{embed-cc-nts} for $n,T$ large.

This completes the proof of Theorem~\ref{T:embedding}. \end{proof}

\section{Existence of minimal blowup solutions}\label{S:exist}

In this section, we show that if Theorem~\ref{T:main}(i) fails, then we may find a minimal blowup solution strictly below the threshold given in Theorem~\ref{T:main}.  Furthermore, as a consequence of minimality, we can show that this solution possesses good compactness properties, namely, its orbit is precompact in $H_x^1(\R^3)$.

\begin{theorem}[Existence of minimal blowup solutions]\label{T:exist} Suppose Theorem~\ref{T:main}(i) fails.  Then there exists $\E_c\in (0,\E_a)$ and a global solution $v$ to \eqref{nls} satisfying:
\begin{itemize}
\item $M(v) = 1$, $\|v(0)\|_{\dot H_a^1} < \K_a$, and $E_a(v) = \E_c$,
\item $v$ blows up in both time directions, in the sense that
\[
\|v\|_{L_{t,x}^5((-\infty,0)\times\R^3)} = \|v\|_{L_{t,x}^5((0,\infty)\times\R^3)} = \infty,
\]
\item $\{v(t)\}_{t\in\R}$ is precompact in $H_x^1(\R^3)$.
\end{itemize}
\end{theorem}

\begin{proof}[Proof of Theorem~\ref{T:exist}] Define
\[
L(\E) : = \sup\{\|u\|_{L_{t,x}^5(I\times\R^3)}\},
\]
where the supremum is taken over all maximal-lifespan solutions $u:I\times\R^3$ such that
\[
M(u)E_a(u) \leq \E\qtq{and} \|u(t)\|_{L_x^2} \|u(t)\|_{\dot H^1_a} < \K_a
\]
for some $t\in I$.

By Theorem~\ref{T:LWP}, Proposition~\ref{P:coercive}, and Remark~\ref{R:coercive}, we have that $L(\E)<\infty$ for all $\E$ sufficiently small; in fact,
\begin{equation}\label{small-data-bds}
L(\E) \lesssim \E^{\frac14}\qtq{for} 0<\E\lesssim \eta_0,
\end{equation}
where $\eta_0$ is the small-data threshold.

Now suppose that Theorem~\ref{T:main}(i) fails. Using Remark~\ref{remark1}, we see that there must exist a `critical' $\E_c\in (0, \E_{a})$ such that
\[
L(\E)<\infty\text{ for } \E<\E_c\qtq{and} L(\E)=\infty\text{ for }\E>\E_c.
\]

As we will see, the key ingredient in the proof of Theorem~\ref{T:exist} is the following proposition.

\begin{proposition}[Palais--Smale condition]\label{P:PS} Fix $a>-\frac14$.  Let $u_n:I_n\times\R^3\to\C$ be a sequence of solutions to \eqref{nls} such that $M(u_n)E_a(u_n) \nearrow \E_c$, and suppose $t_n\in I_n$ satisfy
\begin{gather}
\label{PS-belowQ}
\|u_n(t_n)\|_{L_x^2} \|u_n(t_n)\|_{\dot H^1_a} < \K_a, \\
\nonumber
\lim_{n\to\infty} \|u_n\|_{L_{t,x}^5(\{t< t_n\}\times\R^3)}=\lim_{n\to\infty} \|u_n\|_{L_{t,x}^5(\{t> t_n\}\times\R^3)}=\infty.
\end{gather}
Then, with $\lambda_n := M(u_n)$, we have that $\{u_n^{\lambda_n}(t_n)\}$ converges along a subsequence in $H_x^1$, where we use the notation from \eqref{scaling}.
\end{proposition}

Assuming Proposition~\ref{P:PS} holds for now, we can complete the proof of Theorem~\ref{T:exist} as follows.

If Theorem~\ref{T:main}(i) failed, there would exist a sequence of solutions $u_n:I_n\times\R^3\to\C$ and a sequence of times $\{t_n\}$ satisfying the hypotheses of Proposition~\ref{P:PS}.  By Proposition~\ref{P:PS}, the rescaled solutions $u^{\lambda_n}(t_n)$ converge along a subsequence to some function $v_0\in H_x^1$ with $M(v_0)=1$, $E_a(v_0)=\E_c<\E_a$, and $\|v_0\|_{\dot H^1_a}\leq \K_a$.

Let $v$ be the maximal-lifespan solution to \eqref{nls} with $v(0)=v_0$.  By Remark~\ref{R:coercive}, $v$ is global.  Also, since $\E_c<\E_a$, we can use Proposition~\ref{P:coercive} to see that $\|v(t)\|_{\dot H^1_a}<\K_a$ for all $t\in\R$.  Furthermore, by Theorem~\ref{T:stab}, $v$ must have infinite $L_{t,x}^5$-norm in both time directions.

Finally, to see that the orbit of $v$ is precompact in $H_x^1$, we note that for any $\{\tau_n\}\subset\R$, the sequence $\{v(\tau_n)\}$ satisfies the hypotheses of Proposition~\ref{P:PS} and hence converges along a subsequence. (As $M(v)=1$, no rescaling is necessary.)
\end{proof}

It remains to prove Proposition~\ref{P:PS}.

\begin{proof}[Proof of Proposition~\ref{P:PS}] By Remark~\ref{R:coercive}, we have that $I_n=\R$.  Let $\lambda_n = M(u_n)$ and define $\tilde u_n = u_n^{\lambda_n}$. By time-translation invariance, we may assume $t_n\equiv 0$; thus, we have
\begin{equation}\label{PS-blowup}
\lim_{n\to\infty} \|\tilde u_n\|_{L_{t,x}^5((0,\infty)\times\R^3)} = \lim_{n\to\infty} \|\tilde u_n\|_{L_{t,x}^5((-\infty,0)\times\R^3)} = \infty.
\end{equation}
Note that $\|\tilde u_n(t)\|_{L_x^2}\equiv 1$, $E_a(\tilde u_n)\to \E_c$, and \eqref{PS-belowQ} holds for $\tilde u_n$.

Applying Proposition~\ref{P:LPD} to the sequence $\{\tilde u_n(0)\}$ and passing to a subsequence yields the decomposition
\[
\tilde u_n(0) = \sum_{j=1}^J \phi_n^j + r_n^J \qtq{for all finite} 0 \leq J \leq J^*\in\{0,1,2,\dots,\infty\},
\]
which satisfies the conclusions of Proposition~\ref{P:LPD}. To prove Proposition~\ref{P:PS}, we need to show that $J^*=1$, $r_n^1 \to 0$ in $H_x^1$, $t_n^1\equiv 0$, and $x_n^1\equiv 0$.

We first claim that $\liminf_n E_a(\phi_n^j)>0$ for each $j$. To see this, first note that
\begin{equation}\label{onebub-conv}
x_n^j\to\infty \qtq{implies} \|\phi_n^j\|_{H_a^1} \to \|\phi^j\|_{H_x^1}>0,
\end{equation}
which is a consequence of \eqref{coo4}. The claim now follows from \eqref{decouple1}, \eqref{PS-belowQ}, and Proposition~\ref{P:coercive}a.(iii).  A similar argument shows that $\liminf_n E_a(r_n^J)>0$ for each $J$.  Thus, there are two possible scenarios: 
\begin{align}\label{PS-onebubble}
&\text{either} \quad\sup_j \limsup_{n\to\infty} E_a(\phi_n^j) = \E_c\\
\label{PS-multibubble}
&\text{or}\quad \sup_j \limsup_{n\to\infty} E_a(\phi_n^j) \leq \E_c -
3\delta\qtq{for some}\delta>0.
\end{align}
To complete the proof of Proposition~\ref{P:PS}, we will show (i) if \eqref{PS-onebubble} holds, then we obtain the desired compactness, and (ii) \eqref{PS-multibubble} cannot occur.

Suppose \eqref{PS-onebubble} holds. Then we must have that $J^*=1$, and we can write
\begin{equation}\label{PS-onebubble2}
\tilde u_n(0) = \phi_n + r_n\qtq{with} r_n\to 0 \text{ in }\dot H_x^1.
\end{equation}
We claim that $r_n\to 0$ in $L_x^2$, as well; however, we postpone the proof of this until we have gathered a few other useful facts.

We first show that $x_n\equiv 0$.  If instead  $|x_n|\to\infty$, then we are in a position to apply Theorem~\ref{T:embedding}.  Indeed, if $t_n\equiv 0$, then \eqref{embedding-threshold} follows from \eqref{onebub-conv}, \eqref{PS-onebubble2}, and Corollary~\ref{C:thresholds}.  If instead $t_n\to\pm\infty$, then we use Corollary~\ref{L4-to-zero}, as well.  Thus, by Theorem~\ref{T:embedding}, for all $n$ sufficiently large there exists a global solution $v_n$ to \eqref{nls} with $v_n(0)=\phi_n$ obeying global space-time bounds.  As
\begin{align*}
& \|\tilde u_n(0) - v_n(0)\|_{\dot H_x^1} = \|\tilde u_n(0) - \phi_n \|_{\dot H_x^1}\to 0 \qtq{and} \|\tilde u_n(0) - v_n(0)\|_{L_x^2} \lesssim 1,
\end{align*}
interpolation and Theorem~\ref{T:stab} imply that $\limsup_{n\to\infty}\|\tilde u_n\|_{L_{t,x}^5} \lesssim 1$, contradicting \eqref{PS-blowup}.  We conclude that $x_n\equiv 0$.

We now preclude the possibility that $t_n\to\pm\infty$.  It suffices to rule out the case $t_n\to\infty$. Recalling that $x_n\equiv 0$ and using \eqref{PS-onebubble2}, Strichartz, and monotone convergence, we find that if $t_n\to\infty$, then
\begin{align*}
\|e^{-it\L}\tilde u_n(0)\|_{L_{t,x}^5((0,\infty)\times\R^3)} &\leq \|e^{-it\L}r_n\|_{L_{t,x}^5((0,\infty)\times\R^3)} \\
&\quad +\|e^{-it\L}\phi\|_{L_{t,x}^5((t_n,\infty)\times\R^3)}\to 0\qtq{as}n\to\infty.
\end{align*}
Using the small-data result in Theorem~\ref{T:LWP}, we deduce that $\|\tilde u_n\|_{L_{t,x}^5((0,\infty)\times\R^3)}\to 0$, contradicting \eqref{PS-blowup}.  Thus, we must have that $t_n\equiv 0$.

We have shown that if \eqref{PS-onebubble} holds, then we may write
\begin{equation}\label{ps-sofar}
\tilde u_n(0) = \phi + r_n,\qtq{with} r_n\to 0\qtq{in}\dot H_x^1.
\end{equation}
To obtain the desired compactness, it remains to show that $r_n\to 0$ in $L_x^2$.  Recalling \eqref{decouple1} and the fact that $\|\tilde u_n\|_{L_x^2}\equiv 1$, it suffices to show that $\|\phi\|_{L_x^2}=1$.  If instead $\|\phi\|_{L_x^2}<1$, then the definition of $\E_c$ (the `inductive hypothesis') and the facts that $E_a(\phi)=\E_c$ and $\|\phi\|_{\dot H_a^1}\leq\K_a$ (consequences of \eqref{decouple1} and \eqref{PS-belowQ}) imply that the solution to \eqref{nls} with initial data $\phi$ is global with finite space-time bounds.  However, since $r_n$ is bounded in $L_x^2$ and converges to zero in $\dot H_x^1$, we see that \eqref{ps-sofar} and Theorem~\ref{T:stab} yield finite space-time bounds for the solutions $\tilde u_n$, contradicting \eqref{PS-blowup}. Hence $\|\phi\|_{L_x^2}=1$, and so $r_n\to 0$ in $H_x^1$.

We have just shown that if \eqref{PS-onebubble} holds, then Proposition~\ref{P:PS} follows.  To complete the proof of Proposition~\ref{P:PS}, it therefore remains to rule out the possibility that \eqref{PS-multibubble} holds.

Suppose towards a contradiction that \eqref{PS-multibubble} holds. Recalling that $\liminf E_a(\phi_n^j)>0$ and $\|\tilde u_n(0)\|_{L_x^2}\equiv 1$, and using \eqref{decouple1}, we see that for every finite $J\leq J^*$, we have
\[
M(\phi_n^j) E_a(\phi_n^j)\leq \E_c - 2\delta\qtq{for all} 1\leq j\leq J\qtq{and} n\text{ large.}
\]
Using \eqref{decouple1}, \eqref{PS-belowQ}, and Proposition~\ref{P:coercive}, we also have that
\begin{align}\label{name2}
\|\phi_n^j\|_{L_x^2} \|\phi_n^j\|_{\dot H_a^1} < (1-\delta')\K_a\qtq{for all}1 \leq j\leq J\qtq{and} n\text{ large.}
\end{align}

Arguing as above, if $|x_n^j|\to\infty$ for some $j$, then \eqref{embedding-threshold} holds for $\phi^j$, and hence Theorem~\ref{T:embedding} yields a global solution $v_n^j$ to \eqref{nls} with $v_n^j(0)=\phi_n^j$.

If $x_n^j\equiv 0$ and $t_n^j\equiv 0$ for some $j$, then we take $v^j$ to be the maximal-lifespan solution to \eqref{nls} with $v^j(0) = \phi^j$. If $x_n^j\equiv 0$ and $t_n^j\to\pm\infty$, then we appeal to Theorem~\ref{T:LWP} to find the maximal-lifespan solution $v^j$ to \eqref{nls} that scatters to $e^{-it\L}\phi^j$ in $H_x^1$ as $t\to\pm\infty$. In both cases, we define
\[
v_n^j(t,x) = v^j(t+t_n^j,x).
\]
Each $v_n^j$ is also a maximal-lifespan solution to \eqref{nls}; for $n$ sufficiently large, $t=0$ belongs to the maximal-lifespan, and
\[
\lim_{n\to\infty} \|v_n^j(0) - \phi_n^j\|_{H_a^1} = 0.
\]
In particular, $E_a(v_n^j) \leq \E_c - \delta$ for all $1\leq j\leq J$ and $n$ large.  By the definition of $\E_c$ and \eqref{name2}, we find that each $v_n^j$ is global in time with uniform space-time bounds.  In particular (using Theorem~\ref{T:embedding} for those $j$ for which $|x_n^j|\to\infty$), for any $\eta>0$ we may find $\psi_\eta^j\in C_c^\infty(\R\times\R^3)$ such
\begin{equation}\label{PS-cc}
\|v_n^j-\psi_\eta^j(\cdot-t_n^j,\cdot - x_n^j)\|_{X(\R\times\R^3)}  < \eta
\end{equation}
for all $n$ sufficiently large, where
\begin{align*}
& X \in \{L_{t,x}^5, L_{t,x}^{\frac{10}{3}}, L_t^{\frac{30}{7}} L_x^{\frac{90}{31}}, L_t^{\frac{30}{7}} \dot H_a^{\frac{31}{60},\frac{90}{31}}\}.
\end{align*}

We now construct approximate solutions to \eqref{nls} that asymptotically match $\tilde u_n(0)$, but which have uniform space-time bounds. By Theorem~\ref{T:stab}, this will lead to a contradiction to \eqref{PS-blowup}. For $t\in\R$, define
\begin{align}\nonumber
&u_n^J(t) : = \sum_{j=1}^J v_n^j(t) + e^{-it\L} r_n^J,\qtq{which satisfies}\\
\label{L:approx-data}
&\lim_{n\to\infty}\|u_n^J(0) - \tilde u_n(0) \|_{H_x^1} = 0 \qtq{for any} J.
\end{align}

Our next goal is the following lemma.
\begin{lemma}[Approximate solutions]\label{L:approx} The functions $u_n^J$ satisfy
\begin{align}
\label{PS-approx1}
& \limsup_{n\to\infty} \bigl\{\|u_n^J(0)\|_{H_x^1} + \|u_n^J\|_{L_{t,x}^5} \bigr\}\lesssim  1\qtq{uniformly in} J, \\
\label{PS-approx2}
& \lim_{J\to J^*}\limsup_{n\to\infty} \||\nabla|^{\frac12}\bigl[(i\partial_t -\L)u_n^J + |u_n^J|^2 u_n^J\bigr]\|_{L_{t,x}^{\frac{10}{7}}} = 0,
\end{align}
where the space-time norms are taken over $\R\times\R^3$.
\end{lemma}
Using Lemma~\ref{L:approx} and \eqref{L:approx-data}, Theorem~\ref{T:stab} implies that the solutions $\tilde u_n$ inherit the uniform $L_{t,x}^5$ space-time bounds of the $u_n^J$ for large $n$, contradicting \eqref{PS-blowup}.  Thus, it remains to establish Lemma~\ref{L:approx}.  For this, we exploit the asymptotic orthogonality of the parameters $(t_n^j,x_n^j)$ in \eqref{orthogonal}.

\begin{lemma}[Orthogonality]\label{L:orthogonal} For any $j\neq k$, we have
\begin{align*}
 \|v_n^jv_n^k\|_{L_{t,x}^{\frac52}}+\|v_n^jv_n^k\|_{L_{t,x}^{\frac53}}+ \|v_n^j v_n^k\|_{L_t^{\frac{15}{7}} L_x^{\frac{45}{31}}}+\|(\L)^{\frac{31}{120}} v_n^j (\L)^{\frac{31}{120}} v_n^k\|_{L_t^{\frac{15}{7}} L_x^{\frac{45}{31}}}\to 0
\end{align*}
as $n\to\infty$.
\end{lemma}
\begin{proof}  If we knew that each $v_n^j(t,x)$ were of the form $\psi^j(t-t_n^j,x-x_n^j)$ for some $\psi^j\in C_c^\infty(\R\times\R^3)$, then the result would follow directly from a change of variables and \eqref{orthogonal}. In fact, \eqref{PS-cc} tells us that we may estimate each $v_n^j$ by such a function (in suitable spaces) up to arbitrarily small errors. Using this together with the uniform bounds on the $v_n^j$, the result quickly follows. \end{proof}

\begin{proof}[Proof of Lemma~\ref{L:approx}] First, using \eqref{PS-belowQ} and \eqref{L:approx-data}, we deduce the $H_x^1$-bound in \eqref{PS-approx1}.
From this bound and the decoupling \eqref{decouple1}, we deduce that
\[
\limsup_{n\to\infty}\sum_{j=1}^J \|\phi_n^j\|_{H_x^1}^2 \lesssim 1\qtq{uniformly in }J.
\]
In fact, in view of \eqref{iso}, \eqref{onebub-conv}, and the definition of profiles, this implies
\[
\sum_{j=1}^\infty \|\phi^j\|_{H_x^1}^2 < \infty
\]

Letting $\eta_0>0$ be the small-data threshold of Theorem~\ref{T:LWP} and using \eqref{small-data-bds}, there exists $J_0=J_0(\eta_0)$ such that
\[
\sup_J \limsup_{n\to\infty} \sum_{j=J_0}^J \|v_n^j\|_{S_a^1(\R)}^2 \lesssim \limsup_{n\to\infty}\sum_{j\geq J_0} \|\phi^j\|_{H^1}^2 <\eta_0.
\]
Thus, we deduce that
\begin{equation}\label{PS-summable}
\limsup_{n\to\infty} \sum_{j=1}^J \| v_n^j\|_{S_a^1(\R)}^2 \lesssim 1\qtq{uniformly in}J.
\end{equation}

Next, using Lemma~\ref{L:orthogonal}, Sobolev embedding, and equivalence of Sobolev spaces,
\begin{align*}
\biggl|\ \biggl\|\sum_{j=1}^J v_n^j\biggr\|_{L_{t,x}^5}^5 - \sum_{j=1}^J \|v_n^j\|_{L_{t,x}^5}^5\biggr| &
\lesssim_J \sum_{j\neq k} \|v_n^j\|_{L_{t,x}^5}^3 \|v_n^j v_n^k\|_{L_{t,x}^{\frac52}}\to 0\qtq{as}n\to\infty.
\end{align*}
As $\|e^{-it\L}w_n^J\|_{L_{t,x}^5}\lesssim 1$ uniformly, we may deduce  the $L_{t,x}^5$ bound in \eqref{PS-approx1} from \eqref{PS-summable}.  A similar argument yields
\begin{equation}\label{unJ-sym}
\limsup_{n\to\infty} \|u_n^J\|_{L_{t,x}^{\frac{10}{3}}} \lesssim 1 \qtq{uniformly in}J,
\end{equation}
which will be useful below.

Next, arguing as above, we have for $s\in\{0,\frac{31}{60}\}$,
\begin{align*}
\biggl\|\sum_{j=1}^J v_n^j \biggr\|_{L_t^{\frac{30}{7}} \dot H_a^{s,\frac{90}{31}}}^2 & \lesssim \biggl\| \biggl(\sum_{j=1}^J (\L)^{\frac{s}{2}} v_n^j\biggr)^2\biggr\|_{L_t^{\frac{15}{7}}L_x^{\frac{45}{31}}} \\
& \lesssim \sum_{j=1}^J \|v_n^j\|_{L_t^{\frac{30}{7}} \dot H_a^{s,\frac{90}{31}}}^2 + C_J \sum_{j\neq k} \|(\L)^{\frac{s}{2}}v_n^j(\L)^{\frac{s}{2}} v_n^k\|_{L_t^{\frac{15}{7}} L_x^{\frac{45}{31}}}.
\end{align*}
Using this, Lemma~\ref{L:orthogonal}, and \eqref{PS-summable}, we deduce
\begin{equation}\label{PS-unJ}
\limsup_{n\to\infty} \| u_n^J\|_{L_t^{\frac{30}{7}} H_a^{\frac{31}{60},\frac{90}{31}}} \lesssim 1\qtq{uniformly in }J.
\end{equation}

We are now ready to show \eqref{PS-approx2}. Denoting $F(z) = -|z|^2 z$, we write
\begin{align}\label{enj1}
e_n^J:=(i\partial_t-\L)u_n^J -F(u_n^J) & =  \sum_{j=1}^J F(v_n^j) - F\bigl(\sum_{j=1}^J v_n^j\bigr) \\
\label{enj2}
& \quad + F(u_n^J-e^{-it\L}r_n^J) - F(u_n^J).
\end{align}

We first estimate \eqref{enj1}. Using the fractional product rule, Sobolev embedding, and equivalence of Sobolev spaces, we find
\begin{align*}
\limsup_{n\to\infty}&\| |\nabla|^{\frac{31}{60}} \eqref{enj1}\|_{L_{t,x}^{\frac{10}{7}}}\\
& \lesssim_J \limsup_{n\to\infty} \sum_{j,k,\ell} \| |\nabla|^{\frac{31}{60}}(v_n^j v_n^k v_n^\ell)\|_{L_{t,x}^{\frac{10}{7}}} \\
&\lesssim_J \limsup_{n\to\infty}\sum_{j,k,\ell} \|v_n^j\|_{L_t^{\frac{30}{7}}L_x^{\frac{45}{8}}} \|v_n^k\|_{L_t^{\frac{30}{7}}L_x^{\frac{45}{8}}}\||\nabla|^{\frac{31}{60}} v_n^\ell\|_{L_t^{\frac{30}{7}} L_x^{\frac{90}{31}}} \\
&\lesssim_J \limsup_{n\to\infty}\sum_{j,k,\ell} \| |\nabla|^{\frac12} v_n^j\|_{L_t^{\frac{30}{7}} L_x^{\frac{90}{31}}}\| |\nabla|^{\frac12} v_n^k\|_{L_t^{\frac{30}{7}} L_x^{\frac{90}{31}}}\||\nabla|^{\frac{31}{60}} v_n^\ell\|_{L_t^{\frac{30}{7}} L_x^{\frac{90}{31}}} \lesssim_J 1.
\end{align*}
On the other hand, using Lemma~\ref{L:orthogonal}, we have that
\[
\limsup_{n\to\infty} \|\eqref{enj1}\|_{L_{t,x}^{\frac{10}{7}}}\lesssim_J \limsup_{n\to\infty} \sum_{j\neq k} \|v_n^j v_n^k\|_{L_{t,x}^{\frac52}}\|v_n^j\|_{L_{t,x}^{\frac{10}{3}}} = 0
\]
for all $J$. Thus, by interpolation, we have that
\[
\lim_{J\to J^*} \limsup_{n\to\infty} \||\nabla|^{\frac12}\eqref{enj1}\|_{L_{t,x}^{\frac{10}{7}}}  = 0.
\]

Next, we estimate \eqref{enj2}.  Using \eqref{PS-unJ} and the same spaces as above,
\begin{align*}
\limsup_{n\to\infty}\| |\nabla|^{\frac{31}{60}} F(u_n^J)\|_{L_{t,x}^{\frac{10}{7}}} &\lesssim \limsup_{n\to\infty} \| u_n^J \|_{L_t^{\frac{30}{7}}L_x^{\frac{45}{8}}}^2 \||\nabla|^{\frac{31}{60}} u_n^J \|_{L_t^{\frac{30}{7}} L_x^{\frac{90}{31}}}\\
& \lesssim 1\qtq{uniformly in}J.
\end{align*}
As $ e^{-it\L}r_n^J\in S_a^1(\R)$, this estimate suffices to show that
\[
\limsup_{n\to\infty} \||\nabla|^{\frac{31}{60}}\eqref{enj2}\|_{L_{t,x}^{\frac{10}{7}}} \lesssim 1 \qtq{uniformly in}J.
\]
On the other hand, using Strichartz, \eqref{unJ-sym}, \eqref{PS-approx1}, and \eqref{rnJ}, we can bound
\begin{align*}
&\lim_{J\to J^*}\limsup_{n\to\infty} \|\eqref{enj2}\|_{L_{t,x}^{\frac{10}{7}}} \\
&\quad \lesssim \lim_{J\to J^*}\limsup_{n\to\infty} \|e^{-it\L}r_n^J\|_{L_{t,x}^5} \bigl(\|u_n^J\|_{L_{t,x}^5}+\|r_n^J\|_{\dot H^{\frac12}}\bigr)\bigl(\|u_n^J\|_{L_{t,x}^{\frac{10}{3}}}+\|r_n^J\|_{L_x^2}\bigr) = 0.
\end{align*}
Thus, by interpolation, we have
\[
\lim_{J\to J^*}\limsup_{n\to\infty} \||\nabla|^{\frac12}\eqref{enj2}\|_{L_{t,x}^{\frac{10}{7}}} = 0.
\]
We conclude that \eqref{PS-approx2} holds, which completes the proof of Lemma~\ref{L:approx}.
\end{proof}
As described above, Lemma~\ref{L:approx} together with Theorem~\ref{T:stab} yields a contradiction to \eqref{PS-blowup}. This rules out the scenario \eqref{PS-multibubble} and hence completes the proof of Proposition~\ref{P:PS}. \end{proof}

\section{Preclusion of minimal blowup solutions}\label{S:not-exist}

In this section, we rule out the existence of solutions as in Theorem~\ref{T:exist}, thus completing the proof of Theorem~\ref{T:main}(i).

\begin{theorem}[Preclusion of minimal blowup solutions]\label{T:not-exist} There are no solutions to \eqref{nls} as in Theorem~\ref{T:exist}.
\end{theorem}

\begin{proof} Suppose towards a contradiction that there exists a solution $v$ as in Theorem~\ref{T:exist}.  Let $\delta>0$ such that $\E_c\leq (1-\delta)\E_a$, and take $\eta>0$ to be determined later. By precompactness in $H_x^1$, there exists $R=R(\eta)>1$ such that
\begin{equation}\label{E:tight}
\int_{|x|>R} |u(t,x)|^2 + |\nabla u(t,x)|^2 + |u(t,x)|^4 \,dx < \eta \qtq{uniformly for}t\in\R.
\end{equation}

Note that we also get a uniform \emph{lower} bound on the $H_x^1$-norm of $u(t)$ by compactness and the fact that the solution is not identically zero (indeed, the solution has infinite $L_{t,x}^5$-norm).  Thus, by Proposition~\ref{P:coercive}a.(ii), we have
\begin{equation}\label{E:lb}
\|u(t)\|_{\dot H_a^1}^2 - \tfrac34 \|u(t)\|_{L_x^4}^4 \geq c \|u(t)\|_{\dot H_a^1}^2\gtrsim_u c\qtq{uniformly for}t\in\R
\end{equation}
for some $c=c(\delta,a)>0$.

Define $w_R$ as in Section~\ref{S:virial}.  Choosing $\eta=\eta(u,c)$ sufficiently small, Lemma~\ref{L:virial}, \eqref{E:tight}, and \eqref{E:lb} imply
\begin{align}\label{no name2}
c \lesssim_u \partial_{tt}\int_{\R^3} w_R(x) |u(t,x)|^2 \,dx \qtq{uniformly for} t\in\R.
\end{align}

Using \eqref{virial} and noting that
\[
\biggl|\partial_t \int_{\R^3} w_R(x) |u(t,x)|^2\,dx\biggr| \lesssim R\|u\|_{L_t^\infty H_x^1}^2 \lesssim_u R \qtq{uniformly for}t\in \R,
\]
we can integrate \eqref{no name2} over any interval of the form $[0,T]$ and use the fundamental theorem of calculus to deduce that $cT \lesssim_u R$. Choosing $T$ sufficiently large now yields a contradiction.\end{proof}

\section{Proof of Theorems~\ref{T:threshold}~and~\ref{T:radial}}\label{S:threshold}

In this section, we prove Theorems~\ref{T:threshold}~and~\ref{T:radial}.

For Theorem~\ref{T:threshold}, we prove that for $a>0$, we have failure of uniform space-time bounds at the threshold. Arguing as in \cite{KVZ}, we choose a sequence of solutions with data  equal to translates the ground state for \eqref{nls0}.

\begin{proof}[Proof of Theorem~\ref{T:threshold}] Let
\[
\phi_n(x) := (1-\eps_n) Q_0 (x-x_n),
\]
where $Q_0$ is the ground state for \eqref{nls0}, $\eps_n\to 0$, and $|x_n|\to\infty$.  Using \eqref{poho} and \eqref{coo4}, we can deduce that
\[
M(\phi_n)E_a(\phi_n)\nearrow \E_a\qtq{and} \|\phi_n\|_{L_x^2} \|\phi_n\|_{\dot H_a^1} \nearrow K_a.
\]
Thus, by Theorem~\ref{T:main}, there exist global solutions $u_n$ to \eqref{nls} with $u_n(0)=\phi_n$.

We now define
\[
\tilde u_n(t,x) = (1-\eps_n) e^{it}[\chi_n Q_0](x-x_n),
\]
where $\chi_n$ is as in \eqref{chin}. Note that
\begin{align*}
\|\tilde u_n(0) - u_n(0) \|_{\dot H_x^{\frac12}} = \big\|[(1-\eps_n)\chi_n - 1]Q_0\big\|_{\dot H_x^{\frac12}} \to 0\qtq{as}n\to\infty.
\end{align*}
Note also that
\[
\|\tilde u_n\|_{L_{t,x}^5([-T,T]\times\R^3)} \gtrsim_{Q_0} T\qtq{for any}T>0.
\]

Using the equation $-\Delta Q_0+Q_0-Q_0^3=0$, we can compute
\begin{align}
\nonumber
e_n & := (i\partial_t-\L)\tilde u_n + |\tilde u_n|^2 \tilde u_n \\
\label{thres-e1}
& = e^{it}\bigl[(1-\eps_n)^3\chi_n^3(x-x_n)-(1-\eps_n)\chi_n(x-x_n)\bigr]Q_0^3(x-x_n) \\
\label{thres-e2}
&\quad + (1-\eps_n)e^{it}[Q_0\Delta \chi_n + 2\nabla \chi_n \cdot \nabla Q_0](x-x_n) \\
\label{thres-e3}
&\quad - \tfrac{a}{|x|^2}(1-\eps_n)e^{it}[\chi_n Q_0](x-x_n).
\end{align}
We now claim that for any fixed $T>0$,
\begin{equation}\label{equ:error}
\| |\nabla|^{\frac12} e_n \|_{N([-T,T])} \to 0 \qtq{as} n\to\infty.
\end{equation}

To see this, first recall that $Q_0$ is Schwartz (cf. \cite[Proposition~B.7]{Tao}).  Thus, taking space-time norms over $[-T,T]\times\R^3$, we can first use Sobolev embedding and dominated convergence to estimate
\begin{align*}
&\|\nabla \eqref{thres-e1}\|_{L_t^1L_x^2}\lesssim T  \big[\|\nabla \chi_n\|_{L_x^3}\|Q_0\|_{L_x^8}^3+\|\nabla Q_0\|_{L_x^2}
	\|Q_0\|_{L_x^\infty}^2\big]\lesssim T, \\
& \|\eqref{thres-e1}\|_{L_t^1L_x^2}\lesssim T\|Q_0\|_{L_x^\infty}^2\big\|\bigl[(1-\eps_n)^3\chi_n^3(x-x_n)-(1-\eps_n)
	\chi_n(x-x_n) \bigr]Q_0 \big\|_{L_x^2}\to0
\end{align*}
as $n\to\infty$.  Next,
\begin{align*}
\big\|\nabla \eqref{thres-e2}\big\|_{L_t^1L_x^2}&\lesssim T\bigl\{\|\nabla\Delta \chi_n\|_{L_x^\infty} + \|\Delta\chi_n\|_{L_x^\infty}+ \|\nabla\chi_n\|_{L_x^\infty} \bigr\}\|Q_0\|_{H_x^2} \\
&\lesssim T\bigl\{ |x_n|^{-3}+|x_n|^{-2}+ |x_n|^{-1} \bigr\} \to 0 \qtq{as}n\to\infty, \\
 \|\eqref{thres-e2}\|_{L_t^1 L_x^2} & \lesssim T\bigl\{ \|\Delta\chi_n\|_{L_x^\infty} + \|\nabla \chi_n \|_{L_x^\infty}\bigr\} \| Q_0\|_{ H_x^1}  \\
&\lesssim T\bigl\{ |x_n|^{-2} + |x_n|^{-1} \}\to 0 \qtq{as}n\to\infty.
\end{align*}
Finally, we have
\begin{align*}
\|\langle\nabla\rangle \eqref{thres-e3} \|_{L_t^1 L_x^2}&\lesssim T\bigl\{ \|\tfrac{\chi_n}{|\cdot + x_n|^2}\|_{L_x^\infty}\|\langle\nabla\rangle Q_0\|_{ L_x^2}  + \|\nabla\bigl(\tfrac{\chi_n}{|\cdot + x_n|^2}\bigr)\|_{L_x^\infty}\|Q_0\|_{ L_x^2} \bigr\} \\
&\lesssim T\bigl\{ |x_n|^{-2} + |x_n|^{-3} \}\to 0\qtq{as}n\to\infty.
\end{align*}
Interpolation now gives \eqref{equ:error}.

An application of Theorem~\ref{T:stab} now implies
\[
\| u_n \|_{L_{t,x}^5([-T,T]\times\R^3)} \gtrsim_{Q_0} T.
\]
As $T>0$ was arbitrary, this implies the result. \end{proof}

We next sketch a proof of Theorem~\ref{T:radial}, which considers radial solutions for $a>0$.  Most of the arguments carry over directly from the proof of Theorem~\ref{T:main}, and hence we focus only on the points where the arguments change.

\begin{proof}[Proof of Theorem~\ref{T:radial}] Fix $a>0$.  We construct (non-zero) radial optimizers to the Gagliardo--Nirenberg inequality (restricted to radial functions) satisfying the elliptic equation and Pohozaev identities just as in the proof of Theorem~\ref{T:GN}.  The restriction to radial functions guarantees the requisite compactness via the radial Sobolev embedding. The analogue of Proposition~\ref{P:coercive} (coercivity) follows immediately.  Part (ii) of Theorem~\ref{T:radial} (blowup above the threshold) then follows as in Section~\ref{S:blowup}.

For part (i) of Theorem~\ref{T:radial} (scattering result below the threshold), nearly all of the arguments carry over directly---there is only one delicate point.  Specifically, in the proof of Proposition~\ref{P:PS} (the Palais--Smale condition), we can no longer apply Theorem~\ref{T:embedding} to construct solutions to \eqref{nls} corresponding to profiles living far from the origin. Indeed, because $\E_{a,\text{rad}}>\E_a=\E_0$ and $\K_{a,\text{rad}}>\K_a=\K_0$, these profiles are no longer necessarily below the thresholds for \eqref{nls0}, which would be necessary to invoke the results of \cite{DHR,HR}.  Fortunately, we can show that in the inverse Strichartz inequality, for sequences of radial functions, we may always take $x_n\equiv 0$ (see below). Hence, we never need to apply Theorem~\ref{T:embedding} in order to prove the Palais--Smale condition.

We argue as follows: Recalling the proof of Proposition~\ref{P:IS}, we see that we may take $x_n\equiv 0$ provided we can establish a uniform upper bound for the parameters $x_n$. Recalling that proof and the notation therein, we see that it suffices to prove that
\begin{equation}\label{radial-is}
\|e^{-it\L} f_n \|_{L_{t,x}^5(\R\times\{|x|>\rho\})} \ll \eps(\tfrac{\eps}{A})^{\frac32}\qtq{provided} \rho = C(\tfrac{A}{\eps})^{\frac{15}2}
\end{equation}
for some large constant $C$. Indeed, this implies an upper bound on the $|x_n|$.

To prove \eqref{radial-is}, we use Strichartz, radial Sobolev embedding, and equivalence of Sobolev spaces:
\begin{align*}
\|e^{-it\L} f_n\|_{L_{t,x}^5(\R\times\{|x|>\rho\})} & \lesssim \|e^{-it\L}f_n\|_{L_{t,x}^{\frac{10}{3}}}^{\frac23}
    \|e^{-it\L} f_n \|_{L_{t,x}^\infty(\R\times\{|x|>\rho\})}^{\frac13} \\
& \lesssim \rho^{-\frac13} A^{\frac23} \| |x| e^{-it\L} f_n\|_{L_{t,x}^\infty}^{\frac13} \\
& \lesssim  \rho^{-\frac13} A^{\frac23} \| e^{-it\L}f_n\|_{L_t^\infty L_x^2}^{\frac16} \| e^{-it\L}f_n\|_{L_t^\infty \dot H_a^1}^{\frac16}\lesssim \rho^{-\frac13}A.
\end{align*}
This gives \eqref{radial-is}, which completes the proof of Theorem~\ref{T:radial}.\end{proof}

\end{document}